\documentclass[amssymb,amscd,11pt,20pt]{amsart}
\usepackage{amsmath,amsthm}
\usepackage{amsfonts}
\usepackage{amscd}
\usepackage{amscd}
\usepackage[all]{xy}
 \newtheorem{thm}{Theorem}[section]
 \newtheorem{cor}[thm]{Corollary}
 
 \newtheorem{prop}[thm]{Proposition}
 \theoremstyle{definition}
 \newtheorem{defn}[thm]{Definition}
 \theoremstyle{remark}
 \newtheorem{rem}[thm]{Remark}
 \newtheorem{ejem}[thm]{Example}
 \newtheorem{ejems}[thm]{Examples}

\numberwithin{equation}{subsection}
\newcommand{\OO}{{\mathcal O}}
\newcommand{\M}{{\mathcal M}}

\newcommand{\U}{{\mathcal U}}

\newcommand{\Bc}{{\mathcal B}}
\newcommand{\HH}{{\mathcal H}}
\newcommand{\W}{{\mathcal W}}
\newcommand{\K}{{\mathcal K}}
\newcommand{\Nc}{{\mathcal N}}

\newcommand{\ZZ}{{\mathbb Z}}
\newcommand{\Lc}{{\mathcal L}}

\newcommand{\T}{{\mathcal T}}
\newcommand{\LL}{{\mathbb L}}
\newcommand{\RR}{{\mathbb R}}
\DeclareMathOperator{\ilim}{\underset\to\lim}

\DeclareMathOperator{\Spec}{{Spec}}
\DeclareMathOperator{\Coker}{{Coker}}
\newcommand{\pp}{{\mathfrak p}}
\newcommand{\wh}{\widehat}

\newcommand{\punto}{{\displaystyle \cdot}}
\newcommand{\proda}[2][]{\underset{#2}{\overset{#1}{\prod}}}
\newcommand{\suma}[2][]{\underset{#2}{\overset{#1}{\sum}}}

\newcommand{\enumera}{\begin{enumerate}}
\newcommand{\eenumera}{\end{enumerate}}
\newcommand{\C}{{\mathcal C}}
 
\DeclareMathOperator{\Hom}{{Hom}}
\DeclareMathOperator{\di}{{d}}
\DeclareMathOperator{\Id}{{Id}}

\begin{document}

\title[Finite spaces and schemes]
 {Finite spaces and schemes}

\author{ Fernando Sancho de Salas}
\address{Departamento de Matem\'{a}ticas and Instituto Universitario de F\'isica Fundamental y Matem\'aticas (IUFFyM), Universidad de Salamanca,
Plaza de la Merced 1-4, 37008 Salamanca, Spain}

\email{fsancho@usal.es}

\subjclass[2010]{14-XX,  05-XX, 06-XX}

\keywords{Finite space, quasi-coherent module, scheme}

\thanks {The  author was supported by research project MTM2013-45935-P (MINECO)}

%\thanks{The author was also supported in part by the Research
% Council of Slovenia.}

%\commby{Daniel J. Rudolph}

%%% ----------------------------------------------------------------------

\begin{abstract} A ringed finite space is a ringed space whose underlying topological space is finite. The category of ringed finite spaces contains, fully faithfully, the category of finite topological spaces and the category of affine schemes. Any ringed space, endowed with a finite open covering, produces a ringed finite space. We introduce the notions of schematic finite space and schematic morphism, showing that they behave, with respect to quasi-coherence, like schemes and morphisms of schemes do.  Finally, we construct a fully faithful and essentially surjective functor from a localization of a full subcategory of the category of schematic finite  spaces and schematic morphisms to  the category of quasi-compact and quasi-separated schemes.

\end{abstract}

%%% ----------------------------------------------------------------------
\maketitle
%%% ----------------------------------------------------------------------

\section*{Introduction}

This paper deals with ringed finite spaces and quasi-coherent modules on them. Let us  motivate why these structures deserve some attention, recalling two results (Theorems 1 and 2 below) of \cite{Sancho}. Let $S$ be a topological space and let $\U=\{ U_1,\dots, U_n\}$ be a finite covering by open subsets. Let us consider the following equivalence relation on $S$: we say that $s\sim s'$ if $\U$ does not distinguish $s$ and $s'$; that is, if we denote $U^s=\underset{ U_i\ni s}\cap U_i$, then  $s\sim s'$ iff $U^s=U^{s'}$.  Let $X=S/\negmedspace\sim$ be the quotient set with the topology given by the following partial order: $[s]\leq [s']$ iff $U^s\supseteq U^{s'}$.  This is a finite $T_0$-topological space (i.e., a finite poset) and the quotient map $\pi\colon S\to X$, $s\mapsto [s]$  is continuous.

Assume now that $S$ is a path connected, locally path connected and locally simply connected topological space  and let $\U$ be a finite covering such that the $U^s$ are simply connected. Then:
\medskip

 {\bf Theorem 1.} {\sl  The functors

\[\aligned \left\{\aligned \text{Locally constant sheaves}\\ \text{of abelian groups on $S$}\endaligned \right\} & \overset{\longrightarrow}\leftarrow \left\{ \aligned \text{Locally constant sheaves}\\ \text{of abelian groups on $X$}\endaligned \right\} \\ \M &\to \pi_*\M \\ \pi^*\Nc &\leftarrow \Nc \endaligned \]
are mutually inverse. In other words, $\pi_1(S,s)\to \pi_1(X,\pi(s))$ is an isomorphism between the fundamental groups of $S$ and $X$. Moreover, if the $U^s$ are homotopically trivial, then $\pi\colon S\to X$ is a weak homotopy equivalence, i.e., $\pi_i(S)\to \pi_i(X)$ is an isomorphism for any $i\geq 0$.
}\medskip

Now, if we take the constant sheaf $\ZZ$ on $X$, it turns out that a sheaf of abelian groups on $X$ is locally constant if and only if it is a quasi-coherent $\ZZ$-module. In conclusion, the category of representations of $\pi_1(S)$ on abelian groups is equivalent to the category of quasi-coherent $\ZZ$-modules on the finite topological space $X$.

Assume now that $S$ is a scheme and that the $U^s$ are affine schemes (a $\U$ with this condition exists if and only if $S$ is quasi-compact and quasi-separated). Let $\OO_S$ be the structural sheaf of $S$ and put $\OO=\pi_*\OO_S$, which is a sheaf of rings on $X$.   Now the result is:

\medskip
{\bf Theorem 2.} {\sl Let $S$ be a scheme,  $\U$ a finite covering such that the $U^s$ are affine schemes and $(X,\OO)$ the ringed finite space constructed above. The functors \[\aligned \{\text{Quasi-coherent $\OO_S$-modules} \} & \overset{\longrightarrow}\leftarrow \{\text{Quasi-coherent $\OO$-modules} \} \\ \M &\to \pi_*\M \\ \pi^*\Nc &\leftarrow \Nc \endaligned \]
are mutually inverse, i.e., the category of quasi-coherent modules on $S$ is equivalent to the category of quasi-coherent $\OO$-modules on $X$. Moreover (see Example \ref{coh-schemes}), this equivalence preserves cohomology: for any quasi-coherent module $\M$ on $S$ one has
\[ H^i(S,\M)=H^i(X,\pi_*\M)\]} \medskip

This theorem may be used to prove cohomological results on schemes by proving them on a finite ringed space. For example, one can prove the Theorem of formal functions, Serre's criterion of affineness (see \cite{Sancho2}), flat base change  or Grothendieck's duality in the context of finite ringed spaces (where the proofs are easier)  obtaining those results for schemes as a  particular case. Thus, the  standard hypothesis of separated or semi-separated on schemes may be replaced by the less restrictive hypothesis of quasi-separated. This will be done in  future papers.

Theorems 1 and 2 led us to conclude that it is worthy to  make a  study of ringed finite spaces and of quasi-coherent modules on them. In \cite{Sancho}  ringed finite spaces were studied from the homotopical point of view. Here we make a cohomological study of ringed finite spaces and quasi-coherent sheaves. While in \cite{Sancho} the topological case case was the guide to follow, here is the algebro-geometric case (i.e., schemes). As a very brief resume, we make a study of those ringed finite spaces and morphisms between them that have a good behavior with respect to quasi-coherence. To be more  precise, let us introduce some definitions and results.

By a ringed finite space we mean  a ringed space $(X,\OO_X)$ whose underlying topological space $X$ is finite, i.e. it is a finite topological space endowed with a sheaf of (commutative with unit) rings. It is well known (since Alexandroff) that a finite topological space is equivalent to a finite preordered set, i.e. giving a topology on a finite set is equivalent to giving a preorder relation. Giving a sheaf of rings $\OO_X$ on a finite topological space is equivalent to give, for each point $p\in X$, a ring $\OO_p$, and for each $p\leq q$ a morphism of rings $r_{pq}\colon \OO_p\to\OO_q$, satisfying the obvious relations ($r_{pp}=\Id$ for any $p$ and $r_{ql}\circ r_{pq}=r_{pl}$ for any $p\leq q\leq l$). The category of ringed finite spaces   is a full  subcategory of the category of ringed spaces and it contains (fully faithfully) the category of finite topological spaces (that we shall refer to as ``the topological case'') and the category of affine schemes (see Examples \ref{ejemplos}, (1) and (2)). If $(S,\OO_S)$ is an arbitrary ringed space (a topological space, a differentiable manifold, a scheme, etc) and we take a finite covering $\U=\{ U_1,\dots,U_n\}$ by open subsets, there is a natural associated ringed finite space $(X,\OO_X)$ and a morphism of ringed spaces $S\to X$ (see Examples \ref{ejemplos}, (3));   we say that $X$ is a finite model of $S$.

The first problem one encounters  is that the category of quasi-coherent sheaves on a ringed finite space may be not abelian, because the kernel of a morphism between quasi-coherent modules may fail to be quasi-coherent. Ringed finite spaces where this problem disappears are called finite spaces. More precisely, by a {\it finite space} we mean a ringed finite space $(X,\OO_X)$ such that the morphisms $r_{pq}\colon \OO_p\to\OO_q$ are flat.  Under this flatness assumption, the category of quasi-coherent modules on $X$ is an abelian subcategory of the abelian category of all $\OO_X$-modules. If $\OO_X$ is a sheaf of noetherian rings (i.e., $\OO_p$ is noetherian for any $p\in X$), then this flatness condition is equivalent to say that the structure sheaf $\OO_X$ is coherent. From the point of view of integral functors, the flatness assumption allows to define integral functors between the derived categories of quasi-coherent sheaves (see Corollary \ref{integral-functors}). That is, the category of finite spaces is the most general framework where the theory of integral functors for quasi-coherent sheaves may be developed. Finally, our main examples (i.e. the ``topological case'' and finite models of schemes) satisfy this flatness condition.

Section \ref{Sect-FinSp} is devoted to the study of the main cohomological properties of quasi-coherent sheaves on  finite spaces. The main results are Theorem \ref{qc-of-proj}, that studies the behaviour of quasi-coherence under projections, and Theorem  \ref{graph}, that studies the cohomological structure of the graph of a morphism; these results are essential for the rest of the paper.

However, finite spaces still have a lot of pathologies. Regarding quasi-coherent sheaves, the main problem is that if $f\colon X\to Y$ is a morphism between finite spaces, then $f_*$ does not preserve quasi-coherence in general, even in the most elementary cases as the inclusion of an open subset in a finite space. The second pathology is that the category of finite spaces does not have fibred products, i.e., the flatness assumption does not survive under fibred products.

Sections \ref{Section-Schematic} and \ref{Section-SchematicMorphisms} are devoted to the study of those finite spaces and morphisms which have a good behavior with respect to quasi-coherent sheaves. They are called {\it schematic} finite spaces and {\it schematic morphisms}, because any finite model of a scheme is a schematic space and a finite model of a morphism of schemes is a schematic morphism. For the precise definition, let $(X,\OO_X)$ be a finite space and $\delta\colon X\to X\times X$ the diagonal morphism.  We say that $X$ is {\it schematic} if $\RR\delta_*\OO_X$ is quasi-coherent (i.e. the higher direct images $R^i\delta_*\OO_X$ are quasi-coherent for any $i\geq 0$). The schematic condition is equivalent to the following property: for any open subset $j\colon U\hookrightarrow X$, and any quasi-coherent module $\Nc$ on $U$,  $\RR j_*\Nc$ is quasi-coherent  (Theorem \ref{extension}). In particular, quasi-coherent modules on a schematic space have the extension property  (as it happens with schemes). A more restrictive notion is that of a semi-separated finite space (which is the analog of a semi-separated scheme). They are defined as those schematic spaces such that $R^i\delta_*\OO_X=0$ for $i>0$. We prove that this is equivalent to say that the diagonal morphism $\delta$ is ``affine''. As it happens with schemes, being schematic is a local question  (but being semi-separated is not) and every schematic ``affine'' space is semi-separated (we shall say a few words about affineness at the end of this introduction).

Section \ref{Section-SchematicMorphisms} is devoted to schematic morphisms: Let $f\colon X\to Y$ be a morphism between finite spaces and $\Gamma\colon X\to X\times Y$ its graph. We say that $f$ is schematic if $\RR\Gamma_*\OO_X$ is quasi-coherent. We prove that, if $f$ is schematic, then $\RR f_*\M$ is quasi-coherent for any quasi-coherent module $\M$ on $X$, and the converse is also true if $X$ is schematic (Theorem \ref{sch-preserv-qc2}). The local structure of schematic spaces and morphisms, their behavior  under direct products or compositions,  their structure and properties in the affine case, Stein's factorization and other questions are also treated in sections \ref{Section-Schematic} and \ref{Section-SchematicMorphisms}.  Let $\C_{Schematic}$ be the category of schematic finite spaces and schematic morphisms. This category has the following properties:

- If $f\colon X\to Y$ is a morphism in $\C_{Schematic}$, then $\RR f_*$  preserves quasi-coherence.

- If $X$ is an object of $\C_{Schematic}$, and $U$ is an open subset of $X$, then $U$ is also an object of $\C_{Schematic}$ and the inclusion morphism $j\colon U\hookrightarrow X$ is a morphism in $\C_{Schematic}$.

- $\C_{Schematic}$ is closed under products and graphs; that is, if $X$ and $Y$ belong to $\C_{Schematic}$, then $X\times Y$ belongs to $\C_{Schematic}$, and if $f\colon X\to Y$ is a morphism in $\C_{Schematic}$, then the graph $\Gamma\colon X\to X\times Y$ is a morphism in $\C$.

At the end of section \ref{Section-SchematicMorphisms} we prove that $\C_{Schematic}$ is the biggest subcategory of the category of finite spaces satisfying these conditions (Theorems \ref{characterization-SchematicCategory1} and \ref{characterization-SchematicCategory2}). Finally, we see an important fact: the category of schematic spaces and schematic morphisms has fibered products (Theorem \ref{fiberedproduct}).

The last section (Section \ref{Schematic-Schemes}) deals with the problem of comparing, in categorical terms, schematic finite spaces and schemes. We show that there is a natural functor
\[ \Spec\colon \{ \text{Schematic Finite Spaces}\}\to \{ \text{Ringed Spaces}\}\]
that extends the natural functor associating the affine scheme $\Spec A$ to a ring $A$. We denote by $\C_{Schematic}^{Open}$ the full subcategory of $\C_{Schematic}$ whose objects are those schematic spaces $(X,\OO)$ such that the restriction morphisms $\OO_p\to\OO_q$ are open (any finite model of a scheme has this property). We show that the functor $\Spec$ induces a functor
\[\Spec\colon \C_{Schematic}^{Open}\to \C_{qcqs-Schemes}\]
where $\C_{qcqs-Schemes}$ denotes the category of quasi-compact and quasi-separated schemes. This functor is essentially surjective, but it is not fully faithful. The problem is that the finite models associated to different coverings of a scheme are not isomorphic. To avoid this problem, i.e., to obtain  a fully faithful functor, one needs to localize the category $\C_{Schematic}^{Open}$ by a certain class of morphisms that we have called {\it weak equivalences}.  A weak equivalence is an schematic and affine morphism $f\colon X\to Y$ such that $f_*\OO_X=\OO_Y$. The name is due to the fact that, in the topological case, these morphisms are weak homotopy equivalences (in the topological ordinary sense). The main result is Theorem \ref{embedding-schemes}, that states that there is a fully faithful and essentially surjective functor from the localization of $\C_{Schematic}^{Open}$ by weak equivalences to the category of quasi-compact and quasi-separated schemes. The proof makes use of Grothendieck's faithfully flat descent (\cite{Grothendieck}). We also show the the category of affine schematic spaces, localized by weak equivalences, is equivalent to the category of affine schemes.  The moral is that there are more schematic spaces (after localization) than schemes; while schemes are those ringed spaces obtained by gluing affine schemes along open subschemes, schematic spaces are finite models of a more general notion: ringed spaces obtained by gluing affine schemes along flat monomorphisms. For example, gluing two copies of an affine line along the generic point is not a scheme, but it is obtained from a schematic finite space, via the functor  $\Spec$.

Finally, let us say a few words about affineness. In subsection \ref{Subsect-AffineFinSp} we introduce the notion of an {\it affine} finite space, which is inspired in the algebro-geometric case, i.e., in the characterization of an affine scheme by its quasi-coherent modules. We say that a  finite space $(X,\OO_X)$ is affine if it is acyclic ($H^i(X,\OO_X)=0$ for any $i>0$) and taking global sections gives an equivalence between the category of quasi-coherent $\OO_X$-modules and the category of $A$-modules (with $A=\OO_X(X)$). In the topological case (i.e., $\OO_X=\ZZ$) being affine is equivalent to being homotopically trival. If $(X,\OO_X)$ is a finite model of a scheme $S$, then $(X,\OO_X)$ is affine if and only if $S$ is an affine scheme. Every finite space is locally affine. The main result is Theorem \ref{AffineFinSp}, that gives a cohomological characterization of affine finite spaces. In section \ref{Section-SchematicMorphisms} we study affine schematic spaces, which play the role in the category of schematic spaces that affine schemes do in the category of schemes. Affine schematic morphisms are also treated. A deeper study of affine schematic spaces is done in \cite{Sancho2}.
 
Many of the results and techniques of this paper are  generalizable to Alexandroff spaces (those topological spaces where each point has a minimal open subset containing it), or finite quivers. Instead of dealing with the greatest possible generality, we have preferred to restrict ourselves to finite spaces, as a guiding and fruitful model for other more general situations.

A summary of these results was presented at the talk ``Quasi-coherent modules on finite spaces" in  Warwick EPSRC Symposium: Fourier-Mukai, 34 years on. 

This paper is dedicated to the beloved memory of Prof. Juan Bautista Sancho Guimer{\'a}. I learned from him most of mathematics I know, in particular the use of finite topological spaces in algebraic geometry.

\section{Preliminaries}

In this section we recall elementary facts about finite topological spaces and ringed spaces. The reader may consult \cite{Barmak} for the results on finite topological spaces and \cite{GrothendieckDieudonne} for  ringed spaces.

\subsection{Finite topological spaces}
\medskip

\begin{defn} A finite topological space  is a topological space with a finite number of points.
\end{defn}

Let $X$ be a finite topological space. For each $p\in X$, we shall denote by  $U_p$  the minimum open subset containing  $p$, i.e., the intersection of all the open subsets containing $p$. These $U_p$ form  a minimal base of open subsets.

\begin{defn} A finite preordered set is a finite set with a reflexive and transitive relation (denoted by $\leq$).
\end{defn}

\begin{thm} There is an equivalence between finite topological spaces and finite preordered sets.

\end{thm}

\begin{proof} If $X$ is a finite topological space, we define the relation: $$p\leq q\quad\text{iff}\quad p\in \bar q \quad (\text{i.e., if } q\in U_p) $$
Conversely, if $X$ is a finite preordered set, we define the following topology on  $X$: the closure of a point $p$ is $\bar p=\{ q\in X: q\leq p\}$.
\end{proof}

\begin{rem} \begin{enumerate}
\item The preorder relation defined above does not coincide with that of \cite{Barmak}, by with its inverse. In other words, the topology associated to a preorder that we have defined above is the dual topology that the one considered in op.cit.
\item If $X$ is a finite topological space, then $U_p=\{ q\in X: p\leq q\}$. Hence $X$ has a minimum $p$ if and only if $X=U_p$.
\end{enumerate}
\end{rem}

A map  $f\colon X\to X'$ between finite topological spaces is continuous if and only if it is monotone: for any $p\leq q$, $f(p)\leq f(q)$.

\begin{prop} A finite topological space is  $T_0$ (i.e., different points have different closures) if and only if the relation $\leq$ is antisymmetric, i.e., $X$ is a partially ordered finite set (a finite poset).
\end{prop}

%Any topological space $S$ has an associated $T_0$-topological space, $T_0(S):=$  the quotient of $S$ by the equivalence relation $p\sim q$ iff $\bar p = \bar{q}$. The quotient map $S\to T_0(S)$  is universal for morphisms from $S$ to $T_0$-spaces.

\medskip
\noindent{\it\ \ Dimension}. The dimension of a finite topological space is the maximum of the lengths of the chains of irreducible closed subsets. Equivalently, it is the maximum of the lengths of the  chains of points  $x_0<x_1<\cdots <x_n$.

\medskip
\begin{ejem}\label{covering}{\bf (Finite topological space associated to a finite covering)} (see \cite{Sancho}, Example 1.6). Let $S$ be a topological space and let   $\U=\{U_1,\dots,U_n\}$ be a finite open covering of $S$. Let us consider the following equivalence relation on $S$: we say that $s\sim s'$ if $\U$ does not distinguish $s$ and $s'$, i.e., if we denote $U^s=\underset{U_i\ni s}\cap U_i$, then $s\sim s'$ iff $U^s=U^{s'}$.  Let $X=S/\negmedspace\sim$ be the quotient set with the topology given by the following partial order: $[s]\leq [s']$ iff $U^s\supseteq U^{s'}$.  This is a finite $T_0$-topological space, and the quotient map $\pi\colon S\to X$, $s\mapsto [s]$, is continuous. Indeed, for each $[s]\in X$, one has that $\pi^{-1}(U_{[s]})=U^s$. We shall say that $X$ is the  finite topological space associated to the topological space $S$ and the finite covering $\U$.

This construction is functorial in $(S,\U)$: Let $f\colon S'\to S$ be a continuous map, $\U$ a finite covering of $S$ and $\U'$ a finite covering of $S'$ that is thinner than $f^{-1}(\U)$ (i.e., for each $s'\in S'$, $U^{s'}\subseteq f^{-1}(U^{f(s')})$). If $\pi\colon S\to X$ and $\pi'\colon S'\to X'$ are the associated finite spaces, one has a continuous map $X'\to X$ and a commutative diagram
\[\xymatrix{ S'\ar[r]^f\ar[d]_{\pi'} & S\ar[d]^\pi\\ X'\ar[r] & X.
}\]

%???

A more intrinsic construction of the finite topological space associated to a finite covering is given by spectral methods (see \cite{TeresaSancho}); indeed, let $\T_\U$ be the topology on $S$ generated by $\U$ and $\T_S$ the given topology of $S$. Then $X=\Spec\T_\U$ and the morphism $S\to X$ corresponds to the inclusion $\T_\U\hookrightarrow \T_S$. In general, for any distributive lattice $B$ and any topological space $S$ there is a bijective correspondence between morphisms (of distributive lattices) $B\to\T_S$ and continuous maps $S\to\Spec B$. In fact, given a continuous map $f\colon S\to\Spec B$, it induces a morphism $f^{-1}\colon \T_{\Spec B}\to \T_S$ whose composition with the natural inclusion $B\hookrightarrow \T_{\Spec B}$, gives the correspondent morphism $B\to\T_S$. For any finite topological space $X$ one has a canonical homeomorphism $\Spec\T_X$ is the $T_0$-fication of $X$. These spectral methods are only used in the proof of Theorem \ref{embedding-schemes}.

\end{ejem}

\subsection{Generalities on ringed spaces}
\medskip

\begin{defn} A ringed space is a pair $(X,\OO)$, where $X$ is a topological space and  $\OO$ is a sheaf of (commutative with unit) rings on  $X$. A morphism or ringed spaces $(X,\OO)\to (X',\OO')$ is a pair $(f,f^\#)$, where $f\colon X\to X'$ is a continuous map and $f^\#\colon \OO'\to f_*\OO'$ is a morphism of sheaves of rings (equivalently, a morphism of sheaves of rings $f^{-1}\OO'\to \OO$).
\end{defn}

\begin{defn} Let $\M$ be an $\OO$-module (a sheaf of $\OO$-modules). We say that $\M$ is {\it quasi-coherent} if for each $x\in X$ there exist an open neighborhood  $U$ of $x$ and an exact sequence
\[ \OO_{\vert U}^I \to \OO_{\vert U}^J\to\M_{\vert U}\to 0\] with $I,J$ arbitrary sets of indexes. Briefly speaking, $\M$ is quasi-coherent if it is locally a cokernel of free modules. We say that $\M$ is an $\OO$-module {\it of finite type} if, for each  $x\in X$, there exist an open neighborhood $U$ and an epimorphism
\[ \OO_{\vert U}^n\to\M_{\vert U}\to 0,\] i.e., $\M$ is locally a quotient of a finite free module. We say that $\M$ is an  $\OO$-module {\it of finite presentation} if, for each point $x\in X$, there exist an open neighborhood  $U$ and an exact sequence
\[ \OO_{\vert U}^m\to \OO_{\vert U}^n\to\M_{\vert U}\to 0.\] That is, $\M$ is locally a cokernel of finite free modules. Finally, we say that    $\M$ is {\it coherent} if it is of finite type and for any open subset $U$ and any morphism $\OO_{\vert U}^n\to\M_{\vert U}$, the kernel is an $\OO_{\vert U}$-module of finite type. In other words, $\M$ is coherent if it is of finite type and for every open subset $U$  any submodule of finite type of   $\M_{\vert U}$ is of finite presentation.
\end{defn}

%If $(X,\OO)$ is a punctual ringed space $(*,A)$ then a quasi-coherent module is just an $A$-module, a module of finite type is a finite $A$-module and a coherent module is just a coherent $A$-module.

Let $f\colon X\to Y$ a morphism of ringed spaces. If $\M$ is a quasi-coherent (resp. of finite type) module  on $Y$, then $f^*\M$ is a quasi-coherent (resp. of finite type) module on $X$.

Let $f\colon \M\to\Nc$ be a morphism of $\OO$-modules. If $\M$ and $\Nc$ are quasi-coherent, the cokernel $\Coker f$ is quasi-coherent too, but the kernel may fail to be quasi-coherent.

Direct sums and direct limits of quasi-coherent modules are quasi-coherent. The tensor product of two quasi-coherent modules is also quasi-coherent.

\section{Ringed finite spaces}

Let $X$ be a finite topological space. Recall that we have  a preorder relation \[ p\leq q \Leftrightarrow p\in \bar q \Leftrightarrow U_q\subseteq U_p\]

 Giving a sheaf $F$ of abelian groups (resp. rings, etc) on $X$ is equivalent to giving the following data:

 - An abelian group  (resp. a ring, etc) $F_p$ for each $p\in X$.

 - A morphism of groups (resp. rings, etc) $r_{pq}\colon F_p\to F_q$ for each $p\leq q$, satisfying: $r_{pp}=\Id$ for any $p$, and $r_{qr}\circ r_{pq}=r_{pr}$ for any $p\leq q\leq r$. These $r_{pq}$ are called {\it restriction morphisms}.

Indeed, if $F$ is a sheaf on $X$, then  $F_p$ is the stalk of $F$ at $p$, and it coincides with the sections of $F$ on $U_p$. That is
\[ F_p=\text{ stalk of } F \text{ at } p = \text{ sections of } F \text{ on } U_p:=F(U_p)\]
The morphisms $F_p\to F_q$ are just the restriction morphisms  $F(U_p)\to F(U_q)$.

\begin{ejem} Given a group $G$, the constant sheaf $G$ on $X$ is given by the data: $G_p=G$ for any $p\in X$, and $r_{pq}=\Id$ for any $p\leq q$.
\end{ejem}

\begin{defn} A {\it ringed finite space} is a ringed space $(X,\OO )$ such that  $X$ is a finite topological space.
\end{defn}

By the previous consideration, one has a ring $\OO_p$ for each $p\in X$, and a morphism of rings $r_{pq}\colon \OO_p\to\OO_q$ for each $p\leq q$, such that $r_{pp}=\Id$ for any $p\in X$ and  $r_{ql}\circ r_{pq}=r_{pl}$ for any $p\leq q\leq l$.
\medskip

 Giving a morphism of ringed spaces $(X,\OO)\to (X',\OO')$ between two ringed finite spaces, is equivalent to giving:

-  a continuous (i.e. monotone) map $f\colon X\to X'$,

 -  for each  $p\in X$, a ring homomorphism  $f^\#_p\colon \OO'_{f(p)}\to \OO_p$, such that, for any  $p\leq q$, the diagram (denote $p' =f(p), q'=f(q)$)
\[ \xymatrix{ \OO'_{p'} \ar[r]^{f^\#_{p}} \ar[d]_{r_{p'q'}} & \OO_{p}\ar[d]^{r_{pq}}\\ \OO'_{q'} \ar[r]^{f^\#_{q}}   & \OO_{q}}\] is commutative. We shall denote by $\Hom(X,Y)$ the set of morphisms of ringed spaces between two ringed spaces $X$ and $Y$.

\begin{ejems}\label{ejemplos} \item[$\,\,$(1)] {\it Punctual ringed spaces}. A ringed finite space is called punctual if the underlying topological space has only one element.  The sheaf of rings is then just a ring. We shall denote by  $(*,A)$ the ringed finite space with topological space $\{*\}$ and ring $A$. Giving a morphism of ringed spaces  $(X,\OO)\to (*,A)$ is equivalent to giving a ring homomorphism $A\to \OO(X)$. In particular, the category of punctual ringed spaces is equivalent to the (dual) category of rings, i.e., the category of affine schemes. In other words, the category of affine schemes is a full subcategory of the category of ringed finite spaces, precisely the full subcategory of punctual ringed finite spaces.

Any ringed space $(X,\OO)$ has an associated punctual ringed space $(*,\OO(X))$ and a morphism or ringed spaces $\pi\colon (X,\OO)\to (*,\OO(X))$ which is universal for morphisms from $(X,\OO)$ to punctual spaces. In other words, the inclusion functor
\[i\colon \{\text{Punctual ringed spaces}\} \hookrightarrow \{\text{Ringed spaces}\}\] has a left adjoint: $(X,\OO)\mapsto (*,\OO(X))$. For any $\OO(X)$-module $M$, $\pi^*M$ is a quasi-coherent module on $X$. We sometimes denote $\widetilde M:=\pi^*M$. If $X\to Y$ is a morphism of ringed spaces and $A\to B$ is the induced morphism between the global sections of $\OO_Y$ and $\OO_X$, then, for any $A$-module $M$ one has that $f^*\widetilde M = \widetilde{M\otimes_AB}$.
\medskip

\item[$\,\,$(2)] {\it Finite topological spaces}. Any finite topological space $X$ may be considered as a ringed finite space, taking
the constant sheaf $\ZZ$ as the sheaf of rings. If  $X$ and $Y$ are two finite topological spaces, then giving a morphism of ringed spaces  $(X,\ZZ)\to (Y,\ZZ)$ is just giving a continuous map $X\to Y$. Therefore the category of finite topological spaces is a full subcategory of the category of ringed finite spaces. The (fully faithful) inclusion functor
\[ \aligned \{\text{Finite topological spaces}\} &\hookrightarrow \{\text{Ringed finite spaces} \}\\ X &\mapsto (X,\ZZ)\endaligned\] has a left adjoint, that maps a ringed finite space $(X,\OO)$ to $X$. Of course, this can be done more generally, removing the finiteness hypothesis: the category of topological spaces is a full subcategory of the category of ringed spaces (sending $X$ to $(X,\ZZ)$), and this inclusion has a left adjoint: $(X,\OO)\mapsto X$.
\medskip

\item[$\,\,$(3)] Let $(S,\OO_S)$ be a ringed space (a scheme, a differentiable manifold, an analytic space, ...).
Let $\U=\{U_1,\dots,U_n\}$ be a finite open covering of $S$. Let $X$ be the finite topological space associated to $S$ and $\U$, and $\pi\colon S\to X$ the natural continuous map  (Example \ref{covering}). We have then a sheaf of rings on $X$, namely  $\OO:=\pi_*\OO_S$, so that $\pi\colon (S,\OO_S)\to (X,\OO)$ is a morphism of ringed spaces. We shall say that $(X,\OO)$ is the  {\it ringed finite space associated to the ringed space $S$ and the finite covering $\U$}. This construction is functorial on $(S,\U)$ and on $S$, as in Example \ref{covering}.

\medskip
\item[$\,\,$(4)] {\it Quasi-compact and quasi-separated schemes}. Let $(S,\OO_S)$ be a scheme and $\U=\{U_1,\dots,U_n\}$  a finite open covering of $S$. We say that $\U$ is {\it locally affine} if for each $s\in S$, the intersection $U_s = \underset{s\in U_i}\cap U_i$ is affine. We have the following:

\begin{prop} Let $(S,\OO_S)$ be a scheme. The following conditions are equivalent:
\enumera
\item $S$ is quasi-compact and quasi-separated.
\item $S$ admits a locally affine finite covering $\U$.
\item There exist a finite topological space $X$ and a continuous map $\pi\colon S\to X$ such that $\pi^{-1}(U_x)$ is affine for any $x\in X$.
\eenumera
\end{prop}

\begin{proof} %(1) $\Rightarrow$ (2). Since $S$ is quasi-compact and quasi-separated, we can find  a finite covering  $U_1,\dots, U_n$ of $S$ by affine schemes and a finite covering  $\{ U_{ij}^k\}$  of $U_i\cap U_j$ by affine schemes. Let $\U=\{ U_i, U_{ij}^k\}$ and let us see that it is a locally affine covering of $S$. Let $s\in S$. We have to prove that $U_s$ is affine. If $s$ only belongs to one  $U_i$, then $U_s=U_i$ is affine. If $s$ belongs to more than one $U_i$, let us denote $U_{ij}^s= \underset{s\in U_{ij}^k}\cap U_{ij}^k$. Since $U_{ij}^k$ are affine schemes contained in an affine scheme (for example $U_i$), one has that $U_{ij}^s$ is affine. Now,  $U_s=\underset{i,j}\cap U_{ij}^s$. Put $U_s=U_{i_1j_1}^s\cap \dots \cap U_{i_nj_n}^s$. Replacing each intersection $U_{i_rj_r}^s\cap U_{i_{r+1}j_{r+1}}^s$ by $U_{i_rj_r}^s\cap U_{j_{r}i_{r+1}}^s\cap U_{i_{r+1}j_{r+1}}^s$, we may assume that $j_k=i_{k+1}$, i.e.
%\[ U_s=U_{i_1i_2}^s\cap U_{i_2i_3}^s\cap U_{i_3i_4}^s\cap \dots \cap U_{i_{n-1}i_n}^s\]
%Now, $U_{i_1i_2}^s\cap U_{i_2i_3}^s$ is affine because it is the intersection of two affine subschemes of the affine scheme $U_{i_2}$. Then $U_{i_1i_2}^s\cap U_{i_2i_3}^s\cap U_{i_3i_4}^s$ is affine because $U_{i_1i_2}^s\cap U_{i_2i_3}^s$ and $U_{i_3i_4}^s$ are affine subschemes of the affine scheme $U_{i_3}$. Proceeding this way, one concludes.
%
%(2) $\Rightarrow$ (3). It suffices to take $X$ as the finite topological space associated to $S$ and $\U$.
%
%(3) $\Rightarrow$ (1). $S$ is covered by the affine open subsets $\{ \pi^{-1}(U_x)\}_{x\in X}$, and the intersections $ \pi^{-1}(U_x)\cap  \pi^{-1}(U_{x'})$ are covered by the affine open subsets $\{ \pi^{-1}(U_y)\}_{y\in U_x\cap U_{x'}}$. Hence $S$ is quasi-compact and quasi-separated.

\end{proof}

% Let $(S,\OO_S)$ be a quasi-compact and quasi-separated scheme. It is not difficult to prove that  one can find an
%affine covering $\U=\{U_1,\dots,U_n\}$ such that, for any $s\in S$, the intersection $U_s = \underset{s\in U_i}\cap U_i$ is affine (we say that $\U$ is locally affine). If $X$ is the finite topological space associated to $S$ and $\U$, the continuous map $\pi\colon S\to X$ satisfies that $\pi^{-1}(U_x)$ is affine for any $x\in X$. Conversely, if $S$ is a scheme and there exist a finite topological space $X$ and a continuous map $f\colon S\to X$ such that $f^{-1}(U_x)$ is affine for any $x\in X$, then $S$ is quasi-compact and quasi-separated.

\end{ejems}

\subsection{Fibered products}

Let $X\to S$ and $Y\to S$ be morphisms between ringed finite spaces. The fibered product $X\times_SY$ is the ringed finite space whose underlying topological space is the ordinary fibered product of topological spaces (in other words it is the fibered product set with the preorder given by $(x,y)\leq (x',y')$ iff $x\leq x'$ and $y\leq y'$) and whose sheaf of rings is: if $(x,y)$ is an element of $X\times_SY$ and $s\in S$ is the image of $x$ and $y$ in $S$, then
\[\OO_{(x,y)}=\OO_x\otimes_{\OO_s}\OO_y\] and the morphisms $\OO_{(x,y)}\to \OO_{(x',y')}$ for each $(x,y)\leq (x',y')$ are the obvious ones. For any $(x,y)\in X\times_SY$, one has that $U_{(x,y)}=U_x\times_{U_s}U_y$, with $s$ the image of $x$ and $y$ in $S$.

One has natural morphisms $\pi_X\colon X\times_SY\to X$ and $\pi_Y\colon X\times_SY\to Y$, such that
\[ \aligned \Hom_S(T,X\times_SY)&\to \Hom_S(T,X)\times \Hom_S(T,Y)\\ f&\mapsto (\pi_X\circ f,\pi_Y\circ f)\endaligned\] is bijective.

When $S$ is a punctual space, $S=\{ *,k\}$, the fibered product will be donoted by $X\times_kY$ or simply by $X\times Y$ when $k$ is understood (or irrelevant).   The underlying topological space is the cartesian product $X\times Y$ and the sheaf of rings is given by $\OO_{(x,y)}=\OO_x\otimes_k\OO_y$.

If $f\colon X\to Y$ is a morphism of ringed finite spaces over $k$, the graph $\Gamma_f\colon X\to X\times_kY$ is the morphism of ringed spaces corresponding to the pair of morphisms $\Id\colon X\to X$ and $f\colon X\to Y$.  Explicitly, it is  given by the continuous map $X\to X\times_kY$, $x\mapsto (x,f(x))$, and by the ring homomorphisms $\OO_x\otimes_k \OO_{f(x)}\to \OO_x$ induced by the identity $\OO_x\to\OO_x$ and by the morphisms $\OO_{f(x)}\to\OO_x$ associated with $f\colon X\to Y$.

More generally, if $X$ and $Y$ are ringed finite spaces over a ring finite space $S$ and $f\colon X\to Y$ is a morphism over $S$, the graph of $f$ is the morphism $\Gamma_f\colon X\to X\times_SY$ corresponding to the pair of morphisms $\Id\colon X\to X$ and $f\colon X\to Y$.

\subsection{Quasi-coherent modules}

 Let $\M$ be a sheaf of  $\OO$-modules on a ringed finite space $(X,\OO)$. Thus, for each $p\in X$, $\M_p$ is an $\OO_p$-module and for each  $p\leq q$ one has a morphism of $\OO_p$-modules $\M_p\to\M_q$, hence a morphism of  $\OO_q$-modules
\[\M_p\otimes_{\OO_p}\OO_q\to\M_q\]

\begin{thm}\label{qc} An $\OO$-module $\M$ is quasi-coherent if and only if for any  $p\leq q$ the morphism
\[\M_p\otimes_{\OO_p}\OO_q\to\M_q\]
is an isomorphism.
\end{thm}

\begin{proof} See \cite{Sancho}.
\end{proof}

\begin{ejem} Let $(X,\OO)$ be a ringed finite space, $A=\OO(X)$ and $\pi\colon (X,\OO)\to (*,A)$ the natural morphism. We know that for any $A$-module $M$, $\widetilde M:=\pi^*M$ is a quasi-coherent module on $X$. The explicit stalkwise description of $\widetilde M$ is given by: $(\widetilde M)_x=M\otimes_A\OO_x$.
\end{ejem}

\begin{cor}\label{corqc} Let $X$ be a ringed finite space with a minimum and $A=\Gamma(X,\OO)$. Then the functors
\[\aligned \{\text{Quasi-coherent $\OO$-modules} \} & \overset{\longrightarrow}\leftarrow \{ \text{$A$-modules}\} \\ \M &\to \Gamma(X,\M) \\ \widetilde M &\leftarrow M \endaligned \]
are mutually inverse.
\end{cor}

\begin{proof} Let $p$ be the minimum of $X$. Then $U_p=X$ and for any sheaf $F$ on $X$, $F_p=\Gamma(X,F)$.
If $\M$ is a quasi-coherent module, then for any $x\in X$, $\M_x=\M_p\otimes_{\OO_p}\OO_x$ (Theorem \ref{qc}). That is, $\M$ is univocally determined by its stalk at $p$, i.e., by its global sections.
\end{proof}

\begin{thm}  $\M$ is an $\OO$-module of finite type if and only if:

 - for each  $p\in X$, $\M_p$ is an $\OO_p$-module of finite type,

 - for any  $p\leq q$ the morphism \[\M_p\otimes_{\OO_p}\OO_q\to\M_q\] is surjective.
\end{thm}

\begin{proof} If $\M$ is of finite type, for each  $p$ one has an epimorphism  $\OO_{\vert U_p}^n\to\M_{\vert U_p}\to 0$. Taking, on the one hand,  the stalk at  $p$ tensored by  $\otimes_{\OO_p}\OO_q$, and on the other hand the stalk at $q$, one obtains a commutative diagram
\[\xymatrix {\OO_q^n \ar[r]\ar[d] & \M_p \otimes_{\OO_p}\OO_q \ar[r]\ar[d] & 0 \\ \OO_q^n \ar[r]  & \M_q  \ar[r]  & 0}\]  and one concludes. Conversely, assume that $\M_p$ is of finite type and  $\M_p\otimes_{\OO_p}\OO_q\to\M_q$ is surjective. One has an epimorphism $\OO_p^n\to\M_p\to 0$, that induces a morphism  $\OO_{\vert U_p}^n\to \M_{\vert U_p}$. This is an epimorphism because it is so at the stalk at any  $q\in U_p$.
\end{proof}

\begin{rem} Let $\M$ be an $\OO$-module on a ringed finite space. Arguing as in the latter theorem, one proves that  $ \M_p\otimes_{\OO_p}\OO_q\to\M_q$ is surjective for any $p\leq q$ if and only if $\M$ is locally a quotient of a free module (i.e., for each $p\in X$ there exist an open neighborhood $U$ of $p$ and an epimorphism $\OO_{\vert U}^I\to \M_{\vert U}\to 0$, for some  set of indexes $I$).
\end{rem}

\begin{thm}\label{coherent} An $\OO$-module $\M$ is coherent if and only if:

- for each $p\in X$, $\M_p$ is a coherent $\OO_p$-module.

- for each $p\leq q$, the morphism $\M_p\otimes_{\OO_p}\OO_q\to\M_q$ is surjective.

- for each $p\leq q$ and each sub-$\OO_p$-module of finite type $N$ of $\M_p$, the natural morphism  $N\otimes_{\OO_p}\OO_q\to M_q$ is injective.

%\begin{enumerate} \item Para todo $p$, todo sub-$\OO_p$-m�dulo finito generado de $\M_p$ es de presentaci�n finita.
%\item Para todo $p\leq q$ y todo sub-$\OO_p$-m�dulo finito $N$ de $\M_p$, $N\otimes_{\OO_p}\OO_q$ es un subm�dulo de $\M_q$.
%\end{enumerate}
\end{thm}

\begin{proof} Let $\M$ be a coherent module. By definition, it is of finite type, so $\M_p$ is a finite $\OO_p$-module and $\M_p\otimes_{\OO_p}\OO_q\to\M_q$ is surjective. Let $N$ be a submodule of finite type of  $\M_p$.  $N$ is the image of a morphism  $\OO_p^n\to \M_p$. This defines a morphism  $\OO_{\vert U_p}^n\to\M_{\vert U_p}$, whose kernel  $\K$ is of finite type because  $\M$ is coherent. Taking the the stalk at  $p$ one concludes that $N$ is of finite presentation. Thus, $\M_p$ is a coherent $\OO_p$-module. Moreover one has an exact sequence  $0\to \K_p\to\OO_p^n\to N\to 0$ and for any $q\geq p$ an exact sequence  $  \K_p \otimes_{\OO_p}\OO_q\to\OO_q^n\to N \otimes_{\OO_p}\OO_q\to 0$ and a commutative diagram
\[ \xymatrix{ & \K_p\otimes_{\OO_p}\OO_q \ar[r]\ar[d] & \OO_q^n \ar[r]\ar[d] & N \otimes_{\OO_p}\OO_q \ar[r]\ar[d] & 0
\\ 0\ar[r] &\K_q \ar[r]  & \OO_q^n \ar[r]  & \M_q   &  }  \] The surjectivity of $K_p\otimes_{\OO_p}\OO_q\to \K_q$ implies the injectivity of  $N\otimes_{\OO_p}\OO_q \to\M_q$.

Assume now that $\M$ is a module  satisfying  the conditions. Let $U$ be an open subset and  $\OO_{\vert U}^n\to\M_{\vert U}$ a morphism, whose kernel is denoted by  $\K$. We have to prove that  $\K$ is of finite type. For each $p\in U$, the image, $N$, of $\OO_p^n\to\M_p$ is of finite presentation, because $\M_p$ is coherent; hence,  $\K_p$ is of finite type. For each $q\geq p$ we have a commutative diagram
\[ \xymatrix{ & \K_p\otimes_{\OO_p}\OO_q \ar[r]\ar[d] & \OO_q^n \ar[r]\ar[d] & N \otimes_{\OO_p}\OO_q \ar[r]\ar[d] & 0
\\ 0\ar[r] &\K_q \ar[r]  & \OO_q^n \ar[r]  & \M_q   &  }  \] Now, the injectivity of $N\otimes_{\OO_p}\OO_q \to\M_q$ implies the surjectivity of  $K_p\otimes_{\OO_p}\OO_q\to \K_q$. Hence $\K$ is of finite type and  $\M$ is coherent.
\end{proof}

\begin{thm} $\OO$ is coherent if and only if:
\begin{enumerate} \item For each $p$,   $\OO_p$ is a coherent ring.
\item For any $p\leq q$, the morphism $\OO_p\to\OO_q$ is flat.
\end{enumerate}
\end{thm}

\begin{proof} It is a consequence of the previous theorem and the ideal criterium of flatness.
\end{proof}

\begin{cor}\label{O-coherent} Let $(X,\OO )$ be a ringed finite space of noetherian rings (i.e,  $\OO_p$ is a noetherian ring for any $p\in X$). Then $\OO$ is coherent if and only if for any $p\leq q$ the morphism $\OO_p\to\OO_q$ is flat.
\end{cor}

\subsection{Cohomology}

Let  $X$ be a finite topological space and  $F$ a sheaf of abelian groups on  $X$.

\begin{prop}\label{aciclicity} If $X$ is a finite topological space with a minimum, then $H^i(X,F)=0$ for any sheaf $F$ and any $i>0$. In particular, for any finite topological space one has
\[ H^i(U_p,F)=0\]
for any $p\in X$, any sheaf $F$ and any $i>0$.
\end{prop}

\begin{proof} Let $p$ be the minimum of $X$. Then $U_p=X$ and, for any sheaf $F$, one has $\Gamma(X,F)=F_p$; thus, taking global sections is the same as taking the stalk at $p$, which is an exact functor.
\end{proof}

Let $f\colon X\to Y$ a continuous map between finite topological spaces and $F$ a sheaf on $X$. The i-th higher direct image $R^if_*F$ is the sheaf on $Y$ given by:
\[ [R^if_*F]_y=H^i(f^{-1}(U_y),F)\]

\begin{rem} Let $X,Y$ be two finite topological spaces and $\pi\colon X\times Y\to Y$ the natural projection. If $X$ has a minimum ($X=U_x$), then, for any sheaf $F$ on $X\times Y$, $R^i\pi_*F=0$ for $i>0$, since $(R^i\pi_*F)_y=H^i(U_x\times U_y,F)=0$ by Proposition \ref{aciclicity}. In particular, $H^i(X\times Y, F)=H^i(Y,\pi_*F)$.
\end{rem}

\medskip
\noindent{\it Standard resolution}. Let $F$ be a sheaf on a finite topological space $X$. We define $C^nF$ as the sheaf on $X$  whose sections on an open subset $U$ are
\[ (C^nF)(U)=\proda{U \ni x_0<\cdots <x_n } F_{x_n}\] and whose  restriction morphisms $(C^nF)(U)\to (C^nF)(V)$ for any $V\subseteq U$ are the natural projections.

One has morphisms $d\colon C^nF \to C^{n+1}F$ defined in each open subset $U$ by the formula
\[ (\di a) (x_0<\cdots < x_{n+1})= \suma{0\leq i\leq n} (-1)^i a(x_0<\cdots \wh{x_i}\cdots <x_{n+1}) + (-1)^{n+1} \bar a (x_0<\cdots <x_n)   \] where $\bar a (x_0<\cdots <x_n)$ denotes the image of  $ a (x_0<\cdots <x_n)$ under the morphism $F_{x_n}\to F_{x_{n+1}}$. There is also a natural morphism   $\di\colon F\to C^0F$. One easily checks that $\di^2=0$.

\begin{thm} $C^\punto F$ is a finite and flasque resolution of  $F$.
\end{thm}

\begin{proof} By definition, $C^nF=0$ for $n>\dim X$. It is also clear that  $C^nF$ are flasque. Let us see that
\[ 0\to F\to C^0F \to \cdots\to C^{\dim X}F\to 0\] is an exact sequence.  We have to prove that  $(C^\punto F)(U_p)$ is a resolution of $F(U_p)$. One has a decomposition
\[ (C^nF)(U_p)= \proda{p=x_0<\cdots <x_n } F_{x_n}\times \proda{p<x_0<\cdots <x_n } F_{x_n} = (C^{n-1}F)(U^*_p)\times (C^nF)(U^*_p)\] with $U_p^*:=U_p-\{ p\}$; via this decomposition, the differential  $\di \colon (C^nF)(U_p) \to (C^{n+1}F)(U_p)$ becomes:
\[ \di(a,b)=(b-\di^*a,\di^*b)\] with $\di^*$ the differential of   $(C^\punto F)(U_p^*)$. It is immediate now that every cycle is a boundary.
\end{proof}

This theorem, together with De Rham's theorem (\cite{Godement}, Thm. 4.7.1), yields that the cohomology groups of a sheaf can be computed with the standard resolution, i.e., $H^i_\phi(U,F)=H^i\Gamma_\phi(U,C^\punto F)$, for any open subset $U$ of $X$, any family of supports $\phi$ and any sheaf $F$ of abelian groups on $X$.

\begin{cor} For any finite topological space $X$, any sheaf $F$ of abelian groups on $X$ and any family of supports $\phi$, one has
\[ H^n_\phi(X,F)=0,\quad \text{for any } n>\text{\rm dim} X.\] %Moreover, if $F_p$ is a finitely generated $\ZZ$-module for any $p\in X$, then
%$H^i_\phi(X,F)$ is a finitely generated $\ZZ$-module for any $i\geq 0$.
\end{cor}

Let $\M$ be an $\OO$-module, $U$ an open subset. For each $x\in U$ there is a natural map $\OO_x\otimes_{\OO(U)}\M(U)\to\M_x$. This induces a morphism $(C^n\OO)(U)\otimes_{\OO(U)}\M(U)\to (C^n\M)(U)$ and then a morphism of complexes of sheaves
$(C^\punto\OO)\otimes_\OO\M\to \C^\punto\M$.

\begin{prop}\label{qc-resolution} If $\M$ is quasi-coherent, then $(C^\punto\OO)\otimes_\OO\M\to \C^\punto\M$ is an isomorphism. Moreover, for any $p\in X$ and any open subset $U\subseteq U_p$, one has that
\[ \Gamma(U,C^\punto\OO)\otimes_{\OO_p}\M_p\to \Gamma(U,C^\punto \M)\] is an isomorphism.
\end{prop}

\begin{proof} Since $\M$ is quasi-coherent, for any $x\in U$, the natural map $\OO_x\otimes_{\OO_p}\M_p\to\M_x$ is an isomorphism. Hence, $(C^n\OO)(U)\otimes_{\OO_p}\M_p\to (C^n\M)(U)$ is an isomorphism, so we obtain the second part of the statement. The first part follows from the second, taking $U=U_p$.
\end{proof}

\medskip
\noindent{\it Integral functors.}
\medskip

For any ringed space $(X,\OO_X)$ we shall denote by $D(X)$ the (unbounded) derived category of complexes of $\OO_X$-modules and by $D_{\rm qc}(X)$ the faithful subcategory of complexes with quasi-coherent cohomology. We shall denote by $D(\text{{\rm Qcoh}(X)})$ the derived category of complexes of quasi-coherent $\OO_X$-modules. For a ring $A$, $D(A)$ denotes the derived category of complexes of $A$-modules.

Let $X,X'$ be two ringed finite spaces, and let $\pi\colon X\times X'\to X, \pi'\colon X\times X'\to X'$ be the natural projections. Given an object $\K\in D(X\times X')$, one defines the integral functor of kernel $\K$ by:

\[ \aligned \Phi_\K\colon D(X)&\to D (X')\\ \M &\mapsto \Phi_\K(\M)= \RR \pi'_*( \K\overset\LL\otimes \LL \pi^*\M)\endaligned\]

In general, if we take $\K$ in $D_{\rm qc}(X\times X')$, $\Phi_K$ does not map $D_{\rm qc}(X)$ into $D_{\rm qc}(X')$; the problem is that $\RR \pi'_*$ does not preserve quasi-coherence in general. However, we shall see that, for finite spaces, this holds.

\section{Finite spaces}\label{Sect-FinSp}

\begin{defn} A {\it finite space}  is a ringed finite space $(X,\OO)$ such that for any $p\leq q$ the morphism $\OO_p\to\OO_q$ is flat.
\end{defn}

Any open subset of a finite space is a finite space. The product of two finite spaces is a finite space.

\begin{prop} Let $(X,\OO)$ be a  finite space. Then, the kernel of any morphism between quasi-coherent $\OO$-modules is also quasi-coherent. Moreover, if
\[ 0\to\M'\to\M\to\M''\to 0\] is an exact sequence of  $\OO$-modules and two of them are quasi-coherent, then the third is quasi-coherent too. In particular, the category of quasi-coherent $\OO$-modules on a finite space is an abelian subcategory of the category of $\OO$-modules.
\end{prop}

\begin{proof} It follows easily from Theorem \ref{qc} and the flatness assumption.
\end{proof}

\begin{ejems} \begin{enumerate} \item Let $(X,\OO$ be a noetherian ringed finite space (i.e., $\OO_p$ is a noetherian ring for any $p\in X$). Then $(X,\OO)$ is a finite space if and only if $\OO$ is coherent (Corollary \ref{O-coherent}).
\item Any finite topological space $X$ is a  finite space (with $\OO=\ZZ$), since the restrictions morphisms are the identity.
\item If $X$ is the ringed finite space associated to a (locally affine) finite affine covering of a quasi-compact and quasi-separated scheme $S$ (see Examples \ref{ejemplos}.3 and \ref{ejemplos}.4), then $X$ is a finite space. This follows from the following fact: if $V\subset U$ is an inclusion between two affine open subsets, the restriction morphism $\OO_S(U)\to\OO_S(V)$ is flat.
\end{enumerate}
\end{ejems}

\subsection{Basic cohomological properties of finite spaces}

For this subsection $X$ is a finite space, i.e., a ringed finite space with flat restrictions.

\begin{ejem}\label{coh-schemes} Let $(S,\OO_S)$ be a quasi-compact and quasi-separated scheme and $(X,\OO)$ the finite space associated to a (locally affine) finite affine covering. The morphism $\pi\colon S\to X$ yields an equivalence between the categories of quasi-coherent modules on $S$ and $X$ (see \cite{Sancho}). Moreover, if $\M$ is a quasi-coherent module on $S$, then
\[ H^i(S,\M)=H^i(X,\pi_*\M),\] since, for any $x\in X$, $(R^i\pi_*\M)_x=H^i(\pi^{-1}(U_x),\M)=0$ for $i>0$, because $\pi^{-1}(U_x)$ is an affine scheme. The topological analog is:

Let $S$ be a path connected, locally path connected and locally homotopically trivial topological space and  let $\U=\{ U_1,\dots,U_n\}$ be a (locally homotopically trivial) finite covering of $S$. Let $X$ be the associated finite (topological) space and $\pi\colon S\to X$ the natural continous map. This morphism  yields an equivalence between the categories of locally constant sheaves on $S$ and $X$ (see \cite{Sancho}). Moreover, if $F$ is a locally constant sheaf on $S$, then
\[ H^i(S,F)=H^i(X,\pi_*F),\] since, for any $x\in X$, $(R^i\pi_*F)_x=H^i(\pi^{-1}(U_x),F)=0$ for $i>0$, because $\pi^{-1}(U_x)$ is homotopically trivial.

\end{ejem}

\begin{thm}\label{qc-of-proj} Let $\pi\colon X\times X'\to X$ be the natural projection, with $X$ a  finite space. For any quasi-coherent sheaf $\M$ on $X\times X'$ and any $i\geq 0$,   $R^i\pi_*\M$ is quasi-coherent.
\end{thm}

\begin{proof} Let $p\in X$ and  $\pi'\colon U_p\times X'\to X'$ the natural projection. One has that $(R^i\pi_*\M)_p = H^i(U_p\times X',\M)= H^i(X', \pi'_*(\M_{\vert U_p\times X'}))$. By Theorem \ref{qc}, we have to prove that $$H^i(X',\pi'_*(\M_{\vert U_p\times X'}))\otimes_{\OO_p}\OO_{q}\to H^i(X',\pi''_*(\M_{\vert U_{q}\times X'}))$$ is an isomorphism for any  $p\leq q$, where $\pi''\colon U_{q}\times X'\to X'$ is the natural projection.

Let us denote   $\Nc = \pi'_*(\M_{\vert U_p\times X'}) $ and $\Nc' =\pi''_*(\M_{\vert U_{q}\times X'})$. Since
 $\OO_p\to\OO_{q}$ is flat, it is enough to prove that  $\Gamma (X', C^n \Nc)\otimes_{\OO_p}\OO_q \to \Gamma (X', C^n \Nc')$ is an  isomorphism. For any $x'\in X'$ one has
\[ \Nc_{x'} = \M_{(p,x')}\quad \text{ and }\quad \Nc'_{x'} = \M_{(q,x')}.\] Since $\M$ is quasi-coherent, $\Nc_{x'}\otimes_{\OO_p}\OO_q = \M_{(p,x')}\otimes_{\OO_{(p,x')}}\OO_{(q,x')} = \M_{(q,x')} =\Nc'_{x'}$. From the definition of $C^n$ it follows that $\Gamma (X', C^n \Nc)\otimes_{\OO_p}\OO_q = \Gamma (X', C^n \Nc')$; indeed,
\[ \Gamma (X', C^n \Nc)\otimes_{\OO_p}\OO_q = \proda{x'_0<\dots <x'_n} \Nc_{x'_n}\otimes_{\OO_p}\OO_q = \proda{x'_0<\dots <x'_n} \Nc'_{x'_n} = \Gamma (X', C^n \Nc')\]
\end{proof}

%\begin{thm} (Projection formula) Let $f\colon X\to Y$ a morphism between finite spaces. For any quasi-coherent module $\Nc$ on $Y$ and any module $\M$ on $X$, the natural morphism
%\[ (\RR f_*\M )\overset\LL\otimes  \Nc \to \RR f_*(\M \overset\LL\otimes \LL f^*\Nc)\]
%is an isomorphism.
%\end{thm}
%
%\begin{proof} Let $y\in Y$ and $x\in f^{-1}(U_y)$ (i.e., $f(x)\geq y$). Since $\Nc$ is quasi-coherent, one has
%\[  \M_x \otimes_{\OO_y}\Nc_y= \M_x\otimes_{\OO_{f(x)}}\Nc_y\otimes_{\OO_y}\OO_{f(x)}= \M_x\otimes_{\OO_{f(x)}}\Nc_{f(x)}= (\M\otimes f^*\Nc)_x  . \]
%By definition of the standard resolution $C^\punto$, this equality implies that:
%\[\Gamma(f^{-1}(U_y),C^\punto\M)\otimes_{\OO_y}\Nc_y = \Gamma(f^{-1}(U_y),C^\punto(\M\otimes f^*\Nc))\] Thus
%\[ (*)\qquad\qquad\qquad\qquad (f_*C^\punto\M)\otimes\Nc = f_*C^\punto(\M\otimes f^*\Nc)\qquad\qquad\qquad\qquad\]
%One concludes easily now: in order to see that $(\RR f_*\M )\overset\LL\otimes  \Nc \to \RR f_*(\M \overset\LL\otimes \LL f^*\Nc)$ is an isomorphism, the question is local on $Y$, hence we may assume that $Y=U_y$. If $\Pc^\punto$ is a resolution of $\Nc$ by flat and quasi-coherent modules, then
%\[ (\RR f_*\M )\overset\LL\otimes  \Nc = (f_*C^\punto\M)\otimes\Pc^\punto \overset{(*)}= f_*C^\punto(\M\otimes f^*\Pc^\punto)= \RR f_*(\M \overset\LL\otimes \LL f^*\Nc)\]
%\end{proof}

\begin{thm}\label{graph} Let $f\colon X\to Y$ be a morphism, $\Gamma\colon X\to X\times Y$ its graph and $\pi\colon X\times Y\to X$ the natural projection. For any quasi-coherent module $\M$ on $X$ the natural morphism
\[ \LL\pi^*\M\overset \LL\otimes \RR\Gamma_*\OO_X \to \RR\Gamma_*\M\] is an isomorphism (in the derived category).
\end{thm}

\begin{proof} For any $(x,y)$ in $X\times Y$, let us denote $U_{xy}=U_x\cap f^{-1}(U_y)$. The natural morphism
\[\Gamma(U_{xy},C^\punto\OO_X)\otimes_{\OO_x}\M_x\to \Gamma(U_{xy},C^\punto \M)\] is an isomorphism by Proposition \ref{qc-resolution}.
Now,  $$[\LL\pi^*\M\overset \LL\otimes \RR\Gamma_*\OO_X]_{(x,y)}= \M_x\otimes_{\OO_x}\Gamma(U_{xy},C^\punto\OO_X)$$ because $\Gamma(U_{xy},C^\punto\OO_X)$ is a complex of flat $\OO_x$-modules; on the other hand,  $[\RR\Gamma_*\M]_{(x,y)}= \Gamma(U_{xy},C^\punto \M)$. We are done.
\end{proof}

To conclude with the basic cohomological properties of finite spaces, let us prove two technical results that will be used in sections \ref{Section-Schematic} and \ref{Section-SchematicMorphisms}.

\begin{thm}\label{aciclico} Let  $X$ be a finite space,  $p\in X$ and $U\subset U_p$ an open subset. If $U $ is acyclic  (i.e., $H^i(U\ ,\OO)=0$ for any $i>0$), then
\begin{enumerate}
\item $\OO_p\to \OO(U )$ is flat.
\item For any quasi-coherent module $\M$  on $U_p$,
\[ H^i(U ,\M)=0,\qquad\text{ for } i>0,\] and the natural morphism
\[ \M_p\otimes_{\OO_p}\OO(U )\to \M(U )\] is an isomorphism.

\end{enumerate}
\end{thm}

\begin{proof} By hypothesis, $(C^\punto \OO)(U )$ is a finite resolution of  $\OO(U )$. Moreover, $(C^i\OO)(U)$ is a flat $\OO_p$-module because $X$ is a finite space.  Hence $\OO(U )$ is a flat $\OO_p$-module. For the second part, one has an exact sequence of flat  $\OO_p$-modules
\[ 0\to \OO(U )\to (C^0\OO)(U )\to\cdots\to (C^n\OO)(U )\to 0\] Hence, the sequence remains exact after tensoring by  $\otimes_{\OO_p}\M_p$. One concludes by Proposition \ref{qc-resolution}.
\end{proof}

\begin{prop}\label{preservacion-cuasi} Let $f\colon X\to Y$ be a morphism between finite spaces, $S$  another finite space and $f\times 1\colon X\times S\to Y\times S$ the induced morphism.  If $R^if_*$ preserves quasi-coherence, so does $ R^i(f\times 1)_*$. Consequently, given morphisms $f\colon X\to Y$ and $f'\colon X'\to Y'$, if $\RR f_*$ and $\RR f'_*$ preserve quasi-coherence, then $\RR (f\times f')_*$ preserves quasi-coherence, with $f\times f'\colon X\times X'\to Y\times Y'$.
\end{prop}
\begin{proof} Let $\M$ be a quasi-coherent module on $X\times S$. Let us see that $R^i(f\times 1)_*\M$ is quasi-coherent. Since the question is local, we may assume that  $S=U_s$. We have to prove that the natural morphisms
\[\aligned \left[ R^i(f\times 1)_*\M\right]_{(y,s)}\otimes_{\OO_y}\OO_{y'}&\to [R^i(f\times 1)_*\M]_{(y',s)}\\ [R^i(f\times 1)_*\M]_{(y,s)}\otimes_{\OO_s}\OO_{s'}&\to [R^i(f\times 1)_*\M]_{(y,s')}\endaligned\] are isomorphisms for any $y\leq y'$ in $Y$ and $s'\in U_s$. For the first, we have
\[ [R^i(f\times 1)_*\M]_{(y,s)}=H^i (f^{-1}(U_y)\times U_s,\M) = [R^if_* (\pi_*\M)]_y\] with $\pi\colon X\times U_s\to X$ the natural projection. Since $R^if_* (\pi_*\M)$ is quasi-coherent, one concludes the first  isomorphism. The second follows from the fact that for every open subset $U$ of $X$, the natural morphism  $H^i(U\times U_s,\M)\otimes_{\OO_s}\OO_{s'}\to H^i(U\times U_{s'},\M)$ is an isomorphism, because $R^i\pi'_*(\M_{\vert U\times U_s})$ is quasi-coherent, with $\pi'\colon U\times U_s\to U_s$ the natural projection.

For the consequence, put $f\times f'$ as the composition of $f\times 1$ and $1\times f'$.
\end{proof}

\subsection{Affine finite spaces}\label{Subsect-AffineFinSp}

Let $(X,\OO)$ be an arbitrary ringed space,  $A=\Gamma(X,\OO)$ and $\pi\colon X \to (*,A)$ the natural morphism. Let $\M$ be an $\OO$-module.

\begin{defn} We say that $\M$ is  {\it acyclic} if $H^i(X,\M)=0$ for any $i>0$. We say that $X$ is  {\it acyclic} if $\OO$ is acyclic.
We say that $\M$ is  {\it generated by its global sections} if the natural map $\pi^*\pi_*\M\to \M$ is surjective. In other words, for any $x\in X$, the natural map
\[ M\otimes_A\OO_x\to \M_x,\qquad M=\Gamma(X,\M),\] is surjective.
\end{defn}

If $\M \to \Nc$ is surjective and $\M$ is generated by its global sections, then $\Nc$ too. If $f\colon X\to Y$ is a morphism of ringed spaces and $\M$ is an $\OO_Y$-module generated by its global sections, then $f^*\M$ is generated by its global sections.

\begin{defn} We say that $(X,\OO)$ is an {\it affine ringed space} if it is acyclic and $\pi^*$ (or $\pi_*=\Gamma(X,\quad )$) gives an equivalence between the category of $A$-modules and the category of quasi-coherent $\OO$-modules. We say that $(X,\OO)$ is {\it quasi-affine} if every quasi-coherent $\OO$-module is generated by its global sections.  We say that $(X,\OO)$ is {\it Serre-affine} if every quasi-coherent module is acyclic.
\end{defn}

Obviously, any affine ringed space is quasi-affine. Before we see the basic properties and relations between these concepts on a finite space, let us see some examples for (may be non-finite) ringed spaces. For the proofs, see \cite{Sancho2}.

\begin{ejems}
\begin{enumerate}
\item Let $S$ be a connected, locally path-connected and locally simply connected topological space. Then $(S,\ZZ)$ is an affine ringed space if and only if $S$ is homotopically trivial.
\item  Let $(S,\C^\infty_S)$ be a  Haussdorff differentiable manifold (more generally, a differentiable space) with a countable basis. Then $(S,\C^\infty_S)$ is affine if and only if $S$ is compact.
\item Let $S=\Spec A$ be an affine scheme ($\OO=\widetilde A$ the sheaf of localizations). Then it is affine, quasi-affine and Serre-affine. A quasi-compact and quasi-separated scheme $S$ is affine (in the usual sense of schemes) if and only if it is affine (in our sense) or Serre-affine (Serre's criterion for affineness). A quasi-compact scheme $S$ is quasi-affine if and only if it is an open subset of an affine scheme.
\item Let $(S,\OO_S)$ be a quasi-compact quasi-separated scheme and $\pi \colon S\to X$ the finite space associated to a (locally affine) finite  covering of $S$. Then $S$ is an affine scheme if and only if $X$ is an affine finite space. Even more, an open subset $U$ of $X$ is affine if and only if $\pi^{-1}(U)$ is affine.
\end{enumerate}
\end{ejems}

\begin{prop}\label{U_p-is-affine} If $X$ is a  ringed finite space with a minimum, then it is affine, quasi-affine and Serre-affine. Hence any  ringed finite space is locally affine.
\end{prop}
\begin{proof} It follows from  Corollary \ref{corqc} and Proposition \ref{aciclicity}.
\end{proof}

From now on, assume that $(X,\OO)$ is a finite space.

\begin{thm}\label{AffineFinSp} Let $X$ be a finite space. The following conditions are equivalent:
\begin{enumerate}
\item $X$ is affine.
\item $X$ is acyclic and quasi-affine.
\item $X$ is quasi-affine and Serre-affine.

\end{enumerate}
\end{thm}

\begin{proof} (1) $\Rightarrow$ (2) is immediate. (2) $\Rightarrow$ (3). We have to prove that any quasi-coherent module $\M$ is acyclic. By hypothesis, $\pi^*M\to\M$ is surjective, with  $M=\M(X)$. Since $M$ is a quotient of a free  $A$-module, $\M$ is a quotient of a free $\OO$-module $\Lc$.
We have an exact sequence $0\to \K\to \Lc\to \M\to 0$. Since $X$ is acyclic, $H^i(X,\Lc)=0$ for any $i>0$. Then $H^d(X,\M)=0$, for $d=\dim X$. That is, we have proved that $H^d(X,\M)=0$ for {\it any} quasi-coherent module $\M$. Then $H^d(X,\K)=0$, so $H^{d-1}(X,\M)=0$, for {\it any} quasi-coherent module $\M$;  proceeding in this way we obtain that  $H^i(X,\M)=0$ for any $i>0$ and any quasi-coherent $\M$.

(3) $\Rightarrow$ (1). By hypothesis $X$ is acyclic and  $\pi_*$ is an exact functor over the category of quasi-coherent  $\OO$-modules. Let us see that $M\to \pi_*\pi^*M$ is an isomorphism for any  $A$-module $M$. If $M=A$, there is nothing to say. If $M$ is a free module, it is immediate. Since any $M$ is a cokernel of free modules, one concludes (recall the exactness of $\pi_*$). Finally, let us see that $\pi^*\pi_*\M\to \M$ is an isomorphism for any quasi-coherent $\M$. The surjectivity holds by hypothesis. If $K$ is the kernel, taking $\pi_*$ in the exact sequence
\[ 0\to K\to \pi^*\pi_*\M\to\M\to 0\] and taking into account that  $\pi_*\circ \pi^*=\Id$, we obtain that $\pi_*K=0$. Since $K$ is generated by its global sections, it must be  $K=0$.
\end{proof}

\begin{cor}\label{product-affine} If $X$ and $Y$ are two   affine (resp. quasi-affine, Serre-affine) spaces, then   $X\times Y$ is affine (resp. quasi-affine, Serre-affine). Moreover, if $A=\OO_X(X)$ and $B=\OO_Y(Y)$ are flat $k$-algebras, and $X$, $Y$ are affine, then $\Gamma(X\times_kY,\OO_{X\times_kY})=A\otimes_kB$.
\end{cor}

\begin{proof} Let $\pi\colon X\times Y\to X$ and $\phi\colon U_x\times Y\to Y$ be the natural projections. Let $\M$ be a quasi-coherent module on $X\times Y$ and $\M'=\M_{\vert U_x\times Y}$. Notice that $R^i\phi_*=0$ for any $i>0$ and $R^i\pi_*$, $\phi_*$ preserve quasi-coherence.

Assume that $X$ and $Y$ are Serre-affine.  Since $Y$ is Serre-affine, $R^i\pi_*\M=0$ for $i>0$; indeed, for each $x\in X$, $(R^i\pi_*\M)_x=H^i(U_x\times Y,\M)=H^i(Y,\phi_*\M')=0$. Then $H^i(X\times Y,\M)= H^i(X,\pi_*\M)=0$ for $i>0$, because $X$ is Serre-affine.

Assume that $X$ and $Y$ are  quasi-affine.  Since $Y$ is quasi-affine, the natural morphism $\pi^*\pi_*\M\to\M$ is surjective; indeed, taking the stalk at $(x,y)$, one obtains the morphism $M\otimes_k\OO_y \to (\phi_*\M')_y$, where $M=(\pi_*\M)_x=\Gamma(Y,\phi_*\M')$. Then, it suffices to see that $\pi_*\M$ is generated by its global sections; but this is immediate since $X$ is quasi-affine.

Now, by Theorem \ref{AffineFinSp}, if $X$ and $Y$ are affine, then $X\times Y$ is affine.

Assume also that $A=\OO_X(X)$ and $B=\OO_Y(Y)$ are flat $k$-algebras, and let us prove that $\Gamma(X\times_kY,\OO_{X\times_kY})=A\otimes_kB$.  It suffices to see that the natural map $\OO_X\otimes_k B\to {\pi}_*\OO_{X\times_kY}$ is an isomorphism, where $\OO_X\otimes_k B$ is the sheaf on $X$ defined by $(\OO_X\otimes_k B)(U)=\OO_X(U)\otimes_k B$ (which is a sheaf because $k\to B$ is flat). The question is local on $X$, hence we may assume that $X=U_x$ (notice that $\OO_x$ is a flat $k$-algebra, by Proposition \ref{derivedcat-affine}), and we have to prove that $\Gamma(U_x\times_kY,\OO_{U_x\times_kY})=\OO_x\otimes_k B$. It suffices to see that the natural morphism $\OO_x\otimes_k \OO_Y\to \phi_*\OO_{U_x\times_kY}$ is an isomorphism, which is immediate by taking the stalk at any $y\in Y$.
\end{proof}

In the topological case, being affine is equivalent to being homotopically trivial:

\begin{thm} {\rm (\cite{Quillen}, Proposition 7.8)} Let $X$ be a  connected finite topological space $(\OO=\ZZ$). Then $X$ is affine if and only if $X$ is homotopically trivial.
\end{thm}

\begin{proof} See \cite{Sancho2}.
\end{proof}

\begin{prop}\label{tens-affine} Let $X$ be an affine finite space, $A=\OO(X)$. For any quasi-coherent modules $\M,\M'$ on $X$, the natural morphism $\M(X)\otimes_A\M'(X) \to(\M\otimes_\OO \M')(X)$ is an isomorphism.
\end{prop}
\begin{proof} For any $A$-modules $M,N$ one has an isomorphism $\pi^*M\otimes_\OO \pi^*N\overset\sim \to \pi^*(M\otimes_AN)$. One concludes because $X$ is affine.
\end{proof}

\begin{prop}\label{derivedcat-affine} Let $X$ be an affine finite space, $A=\OO(X)$. Then
\begin{enumerate}
\item For any $p\in X$, the natural map $A\to \OO_p$ is flat. In other words, the functor $$\pi^*\colon \{ A-\text{modules}\}\to \{\OO-\text{modules}\}$$ is exact.
\item The natural (injective) morphism $A\to \underset{p\in X}\prod \OO_p$ is faithfully flat.
\item The natural functors
\[  \xymatrix{ D({\rm\text{\rm Qcoh}}(X))\ar[rr]\ar[rd]_{\Gamma(X,\underline\quad)} &  & D_{\rm\text{\rm qc}}(X)\ar[ld]^{\RR\Gamma(X,\underline\quad)} \\  &  D(A) & }\] are equivalences.
\end{enumerate}
\end{prop}

\begin{proof} (1) Lets us see that $\pi^*$ is exact. It suffices to see that it is left exact. Let $M\to N$ be a injective morphism of $A$-modules and let $\K$ be the kernel of $\pi^*M\to\pi^*N$. Since $X$ is affine, $\pi_*\pi^*=\Id$; hence $\pi_*\K=0$ and then $\K=0$ because $\K$ is quasi-coherent and $X$ is affine.

(2) Since $A\to\OO_p$ is flat, it remains to prove that $\Spec( \underset{p\in X}\prod \OO_p)\to \Spec A$ is surjective. Let $\pp$
be a prime ideal of $A$ and $k(\pp)$ its residue field. Since $X$ is affine, $\pi^*k(\pp)$ is a (non-zero) quasi-coherent module on $X$, hence there exists $p\in X$ such that $(\pi^*k(\pp))_{\vert U_p}$ is not zero. This means that $\OO_p\otimes_A k(\pp)$ is not zero, so the fiber of $\pp$ under the morphism $\Spec\OO_p\to \Spec A$ is not empty.

(3) $\pi_*\colon {\rm Qcoh}(X)\to \{ A-{\rm modules}\}$ is exact because $X$ is affine (hence Serre-affine), and $\pi^*\colon \{ A-{\rm modules}\}\to {\rm Qcoh}(X) $ is exact by (1). Since $X$ is affine, one concludes that $$\pi_*\colon D({\rm\text{\rm Qcoh}}(X))\to D(A)$$ is an equivalence (with inverse $\pi^*$). To conclude, it suffices to see that if $\M^\punto$ is a complex of $\OO$-modules with quasi-coherent cohomology, the natural morphism $ \pi^*\RR\pi_*\M^\punto \to \M^\punto$ is a quasi-isomorphism. Since $\HH^i(\M^\punto)$ are quasi-coherent and $X$ is affine, one has $H^j(X,\HH^i(\M^\punto))=0$ for any $j>0$. Then $H^j(X,\M^\punto)=H^0(X,\HH^j(\M^\punto))$; in other words, $H^j(\RR\pi_*\M^\punto)=\pi_*\HH^j(\M^\punto)$. Then
$$\HH^j(\pi^*\RR\pi_*\M^\punto)\overset{(1)}= \pi^* H^j(\RR\pi_*\M^\punto)=\pi^*\pi_*\HH^j(\M^\punto)$$ and $\pi^*\pi_*\HH^j(\M^\punto)\to \HH^j(\M^\punto)$ is an isomorphism because $\HH^j(\M^\punto)$ is quasi-coherent and $X$ is affine.
\end{proof}

%\begin{cor} If $V\subseteq U$ is an inclusion between affine open subsets of a finite space $X$, then the restriction morphism $\OO(U)\to\OO(V)$ is flat.
%\end{cor}

%\begin{proof} By Theorem \ref{derivedcat-affine}, the composition $\OO(U)\to \OO(V)\to \underset{p\in V}\prod \OO_p$ is flat   and $\OO(V)\to \underset{p\in V}\prod \OO_p$ is faithfully flat. Conclusion follows.
%\end{proof}

\begin{cor}\label{integral-functors} {\rm (1)} Let $f\colon X\to Y$ be a morphism between finite spaces. If $\M^\punto$ is a complex of $\OO_Y$-modules with quasi-coherent cohomology, then $\LL f^*\M^\punto$ is a complex of $\OO_X$-modules with quasi-coherent cohomology. Hence one has a functor
\[\LL f^*\colon D_{\rm\text{\rm qc}}(Y)\to D_{\rm\text{\rm qc}}(X)\]

{\rm (2)} Let $\M^\punto,\Nc^\punto$ be two complexes of $\OO$-modules on a finite space $X$. If $\M$ and $\Nc$ have quasi-coherent cohomology, so does $\M^\punto\overset\LL\otimes\Nc^\punto$. So one has a functor
\[\overset\LL\otimes \colon D_{\rm\text{\rm qc}}(X)\times D_{\rm\text{\rm qc}}(X)\to D_{\rm\text{\rm qc}}(X)\]

{\rm (3)} Let $X$ and $Y$ be two finite spaces and let $\K\in D_{\rm qc}(X\times Y)$. Then, the integral functor $\Phi_\K\colon D(X)\to D(Y)$ maps $D_{\rm qc}(X)$ into $D_{\rm qc}(Y)$.
\end{cor}

\begin{proof} (1) Since being quasi-coherent is a local question, we may assume that $Y$ is affine. Let us denote $\pi\colon Y\to (*,A)$ the natural morphism, with $A=\OO_Y(Y)$. By Proposition \ref{derivedcat-affine},  $\M^\punto \simeq \pi^* M^\punto$ for some complex of $A$-modules $M^\punto$, and then $\LL f^* \M^\punto \simeq \LL\pi_X^* M^\punto$, with $\pi_X=\pi\circ f$; it is clear that $\LL\pi_X^* M^\punto$ has quasi-coherent cohomology (in fact, it is a complex of quasi-coherent modules).

(2) Again, we may assume that $X$ is affine, and then $\M^\punto\simeq \pi^*M^\punto$, $\Nc^\punto\simeq \pi^*N^\punto$. Then
\[ \M^\punto\overset\LL\otimes\Nc^\punto \simeq  \pi^*M^\punto\overset\LL\otimes \pi^*N^\punto \simeq \pi^*( M^\punto\overset\LL\otimes N^\punto).\]

(3) It follows from (1), (2) and Theorem \ref{qc-of-proj}.
\end{proof}

\section{Schematic finite spaces}\label{Section-Schematic}

Let $X$ be a finite space, $\delta\colon X\to X\times_k X$ the diagonal morphism (we shall see below that $k$ is irrelevant).

\begin{defn} We say that a finite space $X$ is  {\it schematic} if  $R^i\delta_*\OO$ is quasi-coherent for any $i\geq 0$ (for the sake of brevity, we shall say that $\RR\delta_*\OO$ is quasi-coherent). We say that $X$ is {\it semi-separated} if $\delta_*\OO$ is quasi-coherent and $R^i\delta_*\OO=0$ for $i>0$.
\end{defn}

For any $p,q\in X$, let us denote $U_{pq}=U_p\cap U_q$, and $\OO_{pq}=\OO (U_{pq})$. One has natural morphisms $\OO_p\to \OO_{pq}$ and $\OO_q\to \OO_{pq}$. Notice that

\begin{equation} \label{delta-qc} (\delta_*\OO)_{(p,q)}=\OO_{pq},\quad (R^i\delta_*\OO)_{(p,q)}= H^i(U_{pq},\OO)\end{equation}

Then one has:

\begin{prop}\label{schematic-cohomological} A finite space $(X,\OO)$ is schematic if and only if for any $p,q\in X$,  any $p'\geq p$  and any $i\geq 0$, the natural morphism
\[ H^i(U_{pq},\OO)\otimes_{\OO_p}\OO_{p'}\to H^i(U_{p'q},\OO) \] is an isomorphism. $X$ is semi-separated if an only if for any $p,q\in X$ and  any $p'\geq p$   the natural morphism \[ \OO_{pq}\otimes_{\OO_p}\OO_{p'}\to \OO_{p'q}  \] is an isomorphism and $U_{pq}$ is acyclic.

In particular, being schematic (or semi-separated) does not depend on $k$.
\end{prop}

\begin{ejems}
\begin{enumerate} \item If $X$ is the finite space associated to a quasi-compact and quasi-separated scheme $S$ and a locally affine covering $\U$, then $X$ is schematic. It is a consequence of the following fact: if $U$ is an affine scheme and $V\subset U$ is a quasi-compact open subset, then for any affine open subset $U'\subset U$, the natural morphism
\[ H^i(V,\OO_S )\otimes_{\OO_S (U)}\OO_S (U')\to H^i(V\cap U',\OO_S)\] is an isomorphism.

\item A finite topological space $X$ (i.e. $\OO=\ZZ$) is schematic if and only if each connected component is irreducible.
\end{enumerate}
\end{ejems}

It is clear that any open subset of a schematic space is also schematic. We shall see that the converse is also true, i.e., being schematic is a local question. Any semi-separated space is obviously schematic. Moreover, we shall see that schematic spaces are locally semi-separated.

\begin{thm}\label{extension} {\rm (of extension)} Let $X$ be a schematic finite space. For any open subset $j\colon U\hookrightarrow X$ and any quasi-coherent module  $\Nc$ on $U$, $\RR j_*\Nc$ is quasi-coherent. In particular, any quasi-coherent module on $U$ is the restriction to $U$ of a quasi-coherent module on  $X$.
\end{thm}

\begin{proof} Let $\delta_U\colon U\to U\times X$ be the graph of $j\colon U\hookrightarrow X$ and $\pi_U$, $\pi_X$ the projections of $ U\times X$ onto $U$ and $X$. By Theorem \ref{graph}, one has an isomorphism $\LL \pi_U^*\Nc \overset\LL \otimes \RR{\delta_U}_*\OO_{\vert U} \overset\sim\to \RR{\delta_U}_*\Nc$. On the other hand, $\RR{\delta_U}_*\OO_{\vert U} = (\RR {\delta}_*\OO)_{\vert U\times X}$, which is quasi-coherent because $X$ is schematic. Hence, $\RR{\delta_U}_*\Nc$ is quasi-coherent. Since $\RR j_*\Nc = \RR {\pi_X}_*\RR{\delta_U}_*\Nc$, we conclude by Theorem \ref{qc-of-proj}.
\end{proof}

\begin{rem} The converse of the theorem also holds; even more: if $\RR j_*\OO_{\vert U}$ is quasi-coherent for any open subset $j\colon U\hookrightarrow X$, then $X$ is schematic.
\end{rem}

\begin{prop}\label{delta-preserves-qc} Let $X$ be a schematic finite space. For any quasi-coherent module   $\M$ on $X$, $\RR\delta_*\M$ is quasi-coherent.
\end{prop}

\begin{proof}  By Theorem \ref{graph}, one has an isomorphism $\LL\pi_X^*\M\overset\LL \otimes \RR\delta_*\OO \simeq \RR\delta_*\M$. One concludes by the quasi-coherence of $\RR\delta_*\OO$.
\end{proof}

\begin{thm}\label{diagonal-semi-sep} Let $X$ be a semi-separated finite space. For any quasi-coherent module   $\M$ on $X$ and any $p,q\in X$, the natural morphism
\[ \M_p\otimes_{\OO_p}\OO_{pq}\to\M_{pq}\qquad (\M_{pq}=\Gamma(U_{pq},\M))\] is an isomorphism. That is, the natural morphism
\[ \pi ^*\M\otimes \delta_*\OO \to \delta_*\M\] is an isomorphism, where $\pi \colon X\times X\to X$ is any of the natural projections.  Moreover, $R^i\delta_*\M=0$ for any $i>0$. \end{thm}

\begin{proof} By Theorem \ref{graph}, one has an isomorphism $\LL\pi^*\M\overset\LL \otimes \RR\delta_*\OO \simeq \RR\delta_*\M$. Now the result follows from the hypothesis $\RR\delta_*\OO =\delta_*\OO$.
\end{proof}

\begin{prop}\label{schematic-product} If $X$ and $Y$ are schematic (resp. semi-separated), then $X\times_k Y$ is schematic (resp. semiseparated).
\end{prop}

\begin{proof} We prove the schematic case and leave the semi-separated case to the reader (which will not be used in the sequel). The diagonal morphism $X\times_k Y\to (X\times_k Y)\times_k (X\times_k Y)=(X\times_k X)\times_k (Y\times_k Y)$ is the composition of the morphisms $\delta_X\times 1\colon X\times_k Y \to (X\times_k X)\times_kY$ and $1\times \delta_Y\colon (X\times_k X)\times_k Y\to (X\times_k X)\times_k (Y\times_k Y)$. One concludes by Propositions \ref{delta-preserves-qc} and \ref{preservacion-cuasi}.
\end{proof}

\begin{prop} A finite space $X$ is schematic if and only if $U_p$ is schematic  for any $p\in X$.
\end{prop}

\begin{proof} If $X$ is schematic, then $U_p$ is schematic. Conversely, assume that $U_p$ is schematic. We have to prove that $\RR\delta_*\OO$ is quasi-coherent. It suffices to see that $(\RR\delta_*\OO)_{\vert U_p\times U_q}$ is quasi-coherent for any $p,q\in X$. Let us denote $\delta'\colon U_{pq}\to U_{pq}\times U_{pq}$ the diagonal morphism of $U_{pq}$ and $i\colon U_{pq}\hookrightarrow U_p$, $j\colon U_{pq}\hookrightarrow U_q$ the inclusions. Then, $(\RR\delta_*\OO)_{\vert U_p\times U_q}= \RR (i\times j)_*\RR\delta'_*\OO_{\vert U_{pq}}$.
Now, $U_{pq}$ is schematic (because $U_p$ is schematic), so $\RR\delta'_*\OO_{\vert U_{pq}}$ is quasi-coherent. Then $\RR (i\times j)_*\RR\delta'_*\OO_{\vert U_{pq}}$ is quasi-coherent by Theorem \ref{extension} (notice that $U_p\times U_q$ is schematic by Proposition \ref{schematic-product}).
\end{proof}

\begin{prop}\label{schematic-affine} An affine finite space  is schematic if and only if it is semi-separated. Consequently, a finite space $X$ is schematic if and only if $U_p$ is semi-separated for any $p\in X$.
\end{prop}

\begin{proof} Let $X$ be an affine and schematic finite space. We have to prove that $R^i\delta_*\OO=0$ for $i>0$. Since  $X\times X$ is affine (Corollary \ref{product-affine}) and $R^i\delta_*\OO$ is quasi-coherent, it suffices to see that  $H^0(X\times X,R^i\delta_*\OO)=0$, $i>0$. Now,  $H^j(X\times X,R^i\delta_*\OO)=0$ for any $j>0$ and any $i\geq 0$, because $R^i\delta_*\OO$ is quasi-coherent and $X\times X$ is affine. Then $H^0(X\times X,R^i\delta_*\OO)=H^i(X,\OO)=0$ for $i>0$, because $X$ is affine.
\end{proof}

\begin{cor}\label{afinsemiseparado} Let $X$ be a schematic and affine finite space. An open subset $U$ of $X$ is affine if and only if it is acyclic.
\end{cor}

\begin{proof} Any affine open subset is acyclic by definition. Conversely, if $U$ is acyclic, by Theorem \ref{AffineFinSp} it is enough to prove that any quasi-coherent module on  $U$ is generated by its global sections. This follows from the fact that any quasi-coherent module on $U$ extends to a quasi-coherent module on $X$ (by Theorem \ref{extension}) and $X$ is affine.
\end{proof}

\begin{cor}\label{localproperties-schematic} Let $X$ be an schematic finite space. Then
\begin{enumerate} \item For any $p\in X$ and any $q,q'\geq p$, the natural morphism $\OO_q\otimes_{\OO_p}\OO_{q'}\to \OO_{qq'}$ is an isomorphism and the natural morphism $\OO_{qq'}\to \underset{t\geq q,q'}\prod\OO_t$ is faithfully flat.
\item For any $p\leq q$ the natural morphism $\OO_q\otimes_{\OO_p}\OO_q\to\OO_q$ is an isomorphism. That is, the morphism of affine schemes $\Spec\OO_q\to\Spec \OO_p$ induced by $\OO_p\to\OO_q$ is a flat monomorphism.
    \end{enumerate}
\end{cor}

\begin{proof} (1) The isomorphism $\OO_q\otimes_{\OO_p}\OO_{q'}\to \OO_{qq'}$ is a consequence of Proposition \ref{schematic-cohomological}. Now, $U_p$ is semi-separated by Proposition \ref{schematic-affine}. Hence $U_{qq'}$ is acyclic, and then affine by Corollary \ref{afinsemiseparado}. By Proposition \ref{derivedcat-affine},  the morphism $\OO_{qq'}\to \underset{t\geq q,q'}\prod\OO_t$ is faithfully flat.

(2) follows from (1), taking $q=q'$.
\end{proof}

\begin{rem} It can be proved that a ringed finite space $(X,\OO)$ satisfying condition (1) of  Corollary \ref{localproperties-schematic} is a finite and schematic space. For a proof, see \cite{Sancho2}.
\end{rem}

\begin{prop}\label{local-structure-affine} Let $X$ be  an affine and schematic finite space, $A=\OO(X)$. For any $p,q\in X$, the natural morphism $\OO_p\otimes_A\OO_q\to\OO_{pq}$ is an isomorphism.
\end{prop}

\begin{proof} Let $\delta\colon X\to X\times_A X$ be the diagonal morphism. It suffices to prove that the natural morphism $\OO_{X\times_AX}\to\delta_*\OO$ is an isomorphism. Since $X\times_A X$ is affine (by Corollary \ref{product-affine}) and $\delta_*\OO$ is quasi-coherent (because $X$ is schematic), it suffices to see that it is an isomorphism after taking global sections. By Corollary \ref{product-affine}, one has $\Gamma(X\times_AX,\OO_{X\times_AX})=A\otimes_A A=A$.
\end{proof}

For a deeper study of affine schematic spaces, see  \cite{Sancho2}.

\section{Schematic morphisms}\label{Section-SchematicMorphisms}

Let $f\colon X\to Y$ be a morphism and $\Gamma\colon X\to X\times Y$ its graph. For each  $x\in X$ and $y\in Y$ we shall denote
\[ U_{xy}:=U_x\cap f^{-1}(U_y)=\Gamma^{-1}(U_x\times U_y),\quad \OO_{xy}:=\Gamma(U_{xy},\OO_X)=(\Gamma_*\OO_X)_{(x,y)}.\]

%\begin{defn} Diremos que un morfismo $f\colon X\to Y$ es localmente ac�clico si para todo $x\in X$ y todo $y\in Y$ el abierto $U_{xy}$ es ac�clico. Esto equivale a decir que $R^i\Gamma_*\OO_X=0$ para $i>0$, es decir, que $\Gamma$ es ac�clico.
%\end{defn}

\begin{defn} We say that a morphism $f\colon X\to Y$ is  {\it schematic} if   $\RR \Gamma_*\OO_X$ is quasi-coherent. This means that
 for any $(x,y)\leq (x',y')$ and any $i\geq 0$, the natural morphism
 \[ H^i(U_{xy},\OO_X)\otimes_{\OO_{(x,y)}}\OO_{(x',y')}\to H^i(U_{x'y'},\OO_X)\] is an isomorphism.
\end{defn}

\begin{defn} We say that $f\colon X\to Y$ is  {\it locally acyclic} if  $\Gamma_*\OO_X$ is quasi-coherent and $R^i\Gamma_*\OO_X=0$ for $i>0$. This last condition means that $U_{xy}$ is acyclic for any $x\in X$, $y\in Y$.
\end{defn}

Obviously, any locally acyclic morphism is schematic. Being schematic is local in $X$: $f\colon X\to Y$ is schematic if and only if $f_{\vert U_x}\colon U_x\to Y$ is schematic for any $x\in X$. If $f\colon X\to Y$ is schematic, then $f^{-1}(U_y)\to U_y$ is schematic for any  $y\in Y$; consequently, if $f\colon X\to Y$ is schematic, then $f\colon U_x\to U_{f(x)}$ is schematic for any $x\in X$. We shall see that the converse is also true if $Y$ is schematic.

\begin{ejems} \begin{enumerate} \item The identity $X\to X$ is schematic (resp. locally acyclic) if and only if $X$ is schematic (resp. semi-separated). A finite space $X$ is schematic (resp. semi-separated) if and only if for every open subset $U$, the inclusion $U\hookrightarrow X$ is an schematic (resp. locally acyclic) morphism.
\item If $Y$ is a punctual space, then any morphism $X\to Y$ is schematic and locally acyclic.
\item Let $f\colon S'\to S$ be a morphism between quasi-compact and quasi-separated schemes, and let $\U, \U'$ be locally affine coverings of $S$ and $S'$ such that $\U'$ is thinner than $f^{-1}(\U)$. Let $X'\to X$ the induced morphism  between the associated finite spaces. This morphism is schematic.
\end{enumerate}
\end{ejems}

In the topological case, one has the following result (whose proof is quite easy and it is omitted because it will not be used in the rest of the paper):

\begin{prop}  Let $f\colon X \to Y$ be a continuous map between finite topological spaces. The following conditions are equivalent.
\begin{enumerate}
\item $f$ is schematic.
\item $f$ is locally acyclic.
\item For any $x\in X$, $y\in Y$, $U_{xy}$ is  non-empty, connected and acyclic.
\end{enumerate}
Moreover, in any of these cases, $Y$ is irreducible (i.e., schematic)  and any generic point of $X$ maps to the generic point of $Y$.
If $X$ and $Y$ are irreducible, then $f$ is schematic if and only if the generic point of $X$ maps to the generic point of $Y$.
\end{prop}

\begin{thm}\label{sch-preserv-qc} Let $f\colon X\to Y$ be a schematic morphism and  $\Gamma\colon X\to X\times Y$ its graph. For any quasi-coherent module  $\M$ on $X$, one has that $\RR\Gamma_*\M$ and $\RR f_*\M$ are quasi-coherent.
\end{thm}

\begin{proof} By Theorem \ref{graph} one has an isomorphism $\LL\pi_X^*\M\overset\LL\otimes \RR\Gamma_*\OO_X\simeq \RR\Gamma_*\M$. Since $\RR\Gamma_*\OO_X$ is quasi-coherent, $\RR\Gamma_*\M$ is also quasi-coherent. Finally, if $\pi_Y\colon X\times Y\to Y$ is the natural projection, the isomorphism $\RR f_*\M \simeq \RR {\pi_Y}_*\RR \delta_*\M$ gives that $\RR f_*\M$ is quasi-coherent, by Theorem \ref{qc-of-proj}.
\end{proof}

\begin{thm}\label{sch-preserv-qc2} Let $X$ be a schematic finite space. A morphism $f\colon X\to Y$ is schematic if and only if $\RR f_*$ preserves quasi-coherence.
\end{thm}

\begin{proof}  The direct part is given by Theorem \ref{sch-preserv-qc}. For the converse, put the graph of $f$, $\Gamma\colon X\to X\times Y$, as the composition of the diagonal morphism $\delta\colon X\to X\times X$ and $1\times f\colon X\times X\to X\times Y$. Now, $\RR \delta_*$ preserves quasi-coherence because $X$ is schematic and $\RR (1\times f)_*$ preserves quasi-coherence by the hypothesis and Proposition \ref{preservacion-cuasi}. We are done.
\end{proof}

\begin{cor} If $f\colon X\to Y$ is a schematic morphism between schematic spaces, then the graph $\Gamma\colon X\to X\times Y$ is a schematic morphism. In particular, if $X$ is a schematic space, then the diagonal morphism $\delta\colon X\to X\times X$ is schematic.
\end{cor}

\begin{proof}  If $f$ is schematic, then $\Gamma$ is schematic by Theorems \ref{sch-preserv-qc} and \ref{sch-preserv-qc2}.
\end{proof}

\begin{prop}\label{composition-good} The composition of schematic (resp. locally acyclic) morphisms is schematic (resp. locally acyclic).
\end{prop}

\begin{proof} We prove the schematic case and leave the locally acyclic case to the reader (since it will not be used in the sequel). Let $f\colon X\to Y$ and $g\colon Y\to Z$ be schematic morphisms,    $h\colon X\to Z$ its composition. The graph  $\Gamma_h\colon X\to X\times Z$ is the composition of the morphisms
\[ X\overset {\Gamma_f}\to X\times Y \overset{1\times\Gamma_g}\to X\times Y\times Z \overset\pi \to X\times Z\] where $\pi\colon X\times Y\times Z \to X\times Z$ is the natural projection. $\RR{\Gamma_f}_*$ and $\RR{\Gamma_g}_*$ preserve quasi-coherence because $f$ and $g$ are schematic,  $\RR(1\times\Gamma_g)_*$ preserves quasi-coherence by Proposition \ref{preservacion-cuasi} and $\RR\pi_*$ preserves quasi-coherence by Theorem \ref{qc-of-proj}. Hence $\RR{\Gamma_h}_*$ preserves quasi-coherence.
\end{proof}

\begin{prop} The product of schematic (resp. locally acyclic) morphisms is schematic (resp. locally acyclic). That is, if  $f\colon X\to Y$ and $f'\colon X'\to Y'$ are schematic (resp. locally acyclic) morphisms, then  $f\times f'\colon X\times X'\to Y\times Y'$ is schematic (resp. locally acyclic).
\end{prop}

\begin{proof} Again, we prove the schematic case and leave the locally acyclic case to the reader. The graph of $f\times f'$  is the composition of the   morphisms $\Gamma_f\times 1\colon X\times X'\to X\times Y\times X'$ and $1\times\Gamma_{f'}\colon X\times Y\times X'\to  X\times Y\times X'\times Y'$. One concludes by Proposition \ref{preservacion-cuasi}.
\end{proof}

\begin{prop}\label{loc-acyclic-affine} Let $X$ and $Y$ be two affine finite spaces. A morphism $f\colon X\to Y$ is schematic if and only if it is locally acyclic.
\end{prop}

\begin{proof} It is completely analogous to the proof of Proposition \ref{schematic-affine}.
\end{proof}

\begin{cor} Let $Y$ be a schematic finite space. A morphism $f\colon X\to Y$ is schematic if and only if, for any $x\in X$, the morphism $f\colon U_x\to U_{f(x)}$ is locally acyclic.
\end{cor}

\begin{proof} If $f\colon X\to Y$ is schematic, then $f\colon U_x\to U_{f(x)}$ is schematic, hence locally acyclic by Proposition \ref{loc-acyclic-affine}. Conversely, assume that $f\colon U_x\to U_{f(x)}$ is locally acyclic for any $x\in X$. Since $Y$ is schematic, $U_{f(x)}\hookrightarrow Y$ is schematic, so the composition  $U_x\to U_{f(x)}\hookrightarrow Y$ is schematic, by Proposition \ref{composition-good}. Hence $f$ is schematic.
\end{proof}

\begin{thm}\label{graphofschematic} Let $f\colon X\to Y$ be a locally acyclic morphism and  $\Gamma\colon X\to X\times Y$ its graph. For any quasi-coherent $\OO_X$-module  $\M$ and any $(x,y)$ in $X\times Y$, the natural morphism
\[ \M_x\otimes_{\OO_x}\OO_{xy}\to\M_{xy}\qquad (\M_{xy}=\Gamma(U_{xy},\M))\] is an isomorphism. In other words, the natural morphism
\[ \pi_X^*\M\otimes \Gamma_*\OO_X\to \Gamma_*\M\] is an  isomorphism, where $\pi_X\colon X\times Y\to X$ is the natural projection.  Moreover, $R^i\Gamma_*\M=0$ for $i>0$.
\end{thm}

\begin{proof} By Theorem \ref{graph}, one has an isomorphism $\LL\pi_X^*\M\overset\LL\otimes \RR\Gamma_*\OO_X\simeq \RR\Gamma_*\M$. One concludes by the hypothesis $\RR\Gamma_*\OO_X =\Gamma_*\OO_X$.
\end{proof}

\subsection{Affine schematic morphisms}

In this subsection all spaces and morphisms are assumed to be schematic.

\begin{defn} Let $f\colon X\to Y$ be a morphism.
\begin{enumerate}
%\item We say that  $\M$ is generated by its global sections over $Y$ if the natural map  $f^*f_*\M\to\M$ is surjective.
%\item We say that $\M$ is $f$-acyclic if $R^if_*\M=0$ for any $i>0$. We say that $X$ is $f$-acyclic if $\OO_X$ is $f$-acyclic. In this case we also say that $f$ is acyclic.
\item We say that $f$ is quasi-affine if for any quasi-coherent module $\M$ on $X$, the natural morphism  $f^*f_*\M\to\M$ is surjective (i.e., every quasi-coherent  $\OO_X$-module is generated by its  sections over $Y$).
\item We say that $f$ is Serre-affine if $R^if_*\M=0$ for any $i>0$ and any quasi-coherent module $\M$ on $X$.
\item We say that   $f$ is affine  if    $f^{-1}(U_y)$ is affine for any $y\in Y$.
\end{enumerate}
\end{defn}

\begin{rem}\label{relative=absolute} If $Y$ is a punctual space, the relative notions coincide with the absolute ones, i.e.: $f$ is affine (resp. Serre-affine, quasi-affine) if and only if $X$ is affine (resp.  Serre-affine, quasi-affine).
\end{rem}

\begin{prop} A morphism $f\colon X\to Y$ is Serre-affine (resp., quasi-affine, affine) if and only if for any $y\in Y$ the morphism $f\colon f^{-1}(U_y)\to U_y$ is Serre-affine (resp. quasi-affine, affine).
\end{prop}

\begin{proof} The only difficulty is to prove the direct statement for Serre-affine and quasi-affine. But it is easy if one takes into account the extension property of quasi-coherent modules on schematic spaces (Theorem \ref{extension}).
\end{proof}

\begin{prop} If $f\colon X\to Y$ and $g\colon Y\to Z$ are  Serre-affine (resp. quasi-affine) then $g\circ f$   is Serre-affine (resp. quasi-affine).
\end{prop}
\begin{proof} Let us denote $h=g\circ f$. If $f$ and $g$ are Serre-affine, then $R^j g_*R^if_*\M=0$ for any $i+j>0$. Hence $R^ih_*\M=0$ for any $i>0$. If $f$ and $g$ are quasi-affine, then $f^*f_*\M\to \M$ and $g^*g_*(f_*\M)\to f_*\M$ are surjective; hence the composition
\[ h^*h_*\M= f^*g^*g_*f_*\M\to f^*f_*\M\to\M\] is also surjective.
\end{proof}

\begin{prop} A morphism $f\colon X\to Y$ is Serre-affine (resp., quasi-affine, affine) if and only if $f^{-1}(U_y)$ is Serre-affine (resp. quasi-affine), for any $y\in Y$ the morphism
\end{prop}
\begin{proof} If $f$ is Serre-affine, then the composition $f^{-1}(U_y)\to U_y \to (*,\OO_y)$ is Serre-affine, hence $f^{-1}(U_y)$ is Serre-affine. Conversely, if $f^{-1}(U_y)$ is Serre-affine, then $(R^if_*\M)_y=H^i(f^{-1}(U_y),\M)=0$ for any $i>0$, hence $f$ is Serre-affine. For quasi-affine the argument is analogous.
\end{proof}

Now the following is immediate (in view of Theorem \ref{AffineFinSp}):

\begin{thm} Let $f\colon X\to Y$ be a schematic morphism between schematic finite spaces. The following conditions are equivalent:
\begin{enumerate}
\item $f$ is affine.
\item $f$ is quasi-affine and acyclic (i.e., $R^if_*\OO_X=0$ for any $i>0$).
\item $f$ is Serre-affine and quasi-affine.
\end{enumerate}
\end{thm}

\begin{cor} The composition of affine morphisms is affine.
\end{cor}

\begin{prop} Assume that $Y$ is affine. Then a morphism $f\colon X\to Y$ is affine if and only if $X$ is affine.
\end{prop}

\begin{proof} If $f$  is affine, the composition $X\to Y\to (*,\OO_Y(Y))$ is affine, hence $X$ is affine. Conversely, if $X$ is affine, let us prove that $f^{-1}(U_y)$ is affine for any $y\in Y$. It suffices to see that $f^{-1}(U_y)$ is acyclic, so we conclude if we prove  that $R^if_*\OO_X=0$ for $i>0$. Since $Y$ is affine, $H^j(Y,R^if_*\OO_X)=0$ for any $j>0$ and any $i\geq 0$. Hence $H^0(Y, R^if_*\OO_X)=H^i(X,\OO_X)=0$ for $i>0$, because $X$ is affine. But $H^0(Y, R^if_*\OO_X)=0$ implies that $R^if_*\OO_X=0$ because $Y$ is affine.
\end{proof}

\begin{prop} The product of affine (resp. Serre-affine, quasi-affine) morphisms is affine (resp. Serre-affine, quasi-affine). That is, if $f\colon X\to Y$ and $f'\colon X'\to Y'$ are affine (resp., ...), then $f\times f'\colon X\times X'\to Y\times f'$ is affine (resp., ...).
\end{prop}

\begin{proof} For any $(y,y')\in Y\times Y'$, $(f\times f')^{-1}(U_y\times U_{y'})= f^{-1}(U_y) \times {f'}^{-1}(U_{y'})$, which is a product of affine (resp., ...) spaces, hence affine (resp., ...) by Corollary \ref{product-affine}.
\end{proof}

The following result justifies the name ``semi-separated'':

\begin{prop} A schematic finite space $X$ is semi-separated if and only if $\delta\colon X\to X\times X$ is affine.
\end{prop}

\begin{proof} If $X$ is semi-separated, then  $\delta^{-1}(U_p\times U_q)=U_{pq}$ is acyclic, hence affine by Corollary \ref{afinsemiseparado}, since it is contained in $U_p$, schematic and affine. Thus, $\delta$ is affine. Conversely, if $\delta $ is affine, then it is acyclic, so $X$ is semi-separated.
\end{proof}

\begin{prop}\label{loc-acyc=Serre-afin}   A schematic morphism $f\colon X\to Y$ is locally acyclic if and only if its graph $\Gamma\colon X\to X\times Y$ is affine.
\end{prop}

\begin{proof} If $\Gamma$ is affine, then $\Gamma$ is  acyclic, so $f$ is locally acyclic. Conversely, if $f$ is locally acyclic, then  $U_{xy}$ is acyclic and, by Corollary \ref{afinsemiseparado}, $U_{xy}$ is affine, since  $U_x$ is schematic and affine. Hence $\Gamma$ is affine.
\end{proof}

\begin{thm}\label{projection-formula} Let $f\colon X\to Y$ be an affine morphism.

(1) For any quasi-coherent module $\M$ on $X$ and any quasi-coherent module $\Nc$ on $Y$, the natural morphism $\Nc\otimes f_*\M\to f_*(f^*\Nc \otimes \M)$ is an isomorphism.

(2) A morphism $\M\to\M'$ between quasi-coherent modules on $X$ is an isomorphism if and only if the induced morphism $f_*\M\to f_*\M'$ is an isomorphism.

(3) If $f_*\OO_X=\OO_Y$, then $f_*$ and $f^*$ yield an equivalence between the categories of quasi-coherent modules on $X$ an $Y$.

\end{thm}

\begin{proof} (1) Let $y\in Y$ and let us denote $N=\Nc_y$, $M=\M(f^{-1}(U_y))$, $A=\OO_y$, $B=\OO_X(f^{-1}(U_y))$. Taking the stalk at $y$, we obtain the morphism $N\otimes_A M \to \Gamma(f^{-1}(U_y), f^*\Nc\otimes\M)$. By Proposition \ref{tens-affine}, $\Gamma(f^{-1}(U_y), f^*\Nc\otimes\M)= \Gamma(f^{-1}(U_y), f^*\Nc)\otimes_B  M$. Finally, $\Gamma(f^{-1}(U_y), f^*\Nc)=N\otimes_AB$, because $f^{-1}(U_y)$ is affine. Conclusion follows.

(2) For each $y\in Y$, let us denote $X_y=f^{-1}(U_y)$. Now, $\M\to \M'$ is an isomorphism if and only if $\M_{\vert X_y}\to \M'_{\vert X_y}$ is an isomorphism for any $y$. Since $X_y$ is affine, this is an isomorphism if and only if it is an isomorphism after taking global sections, i.e., iff $(f_*\M)_y\to (f_*\M')_y$ is an isomorphism.

(3) For any quasi-coherent module $\Nc$ on $Y$, the natural morphism $\Nc\to f_*f^*\Nc$ is an isomorphism by (1) and the hypothesis $f_*\OO_X=\OO_Y$. For any quasi-coherent module $\M$ on $X$, the natural morphism $f^*f_*\M\to \M$ is an isomorphism by (2), since $f_*(f^*f_*\M)=f_*\M\otimes f_*\OO_X$ and $f_*\OO_X=\OO_Y$.
\end{proof}

\begin{ejem}\label{ejem-affine} Let $(X,\OO)$ be a finite space and $\Bc$ a quasi-coherent $\OO$-algebra; that is, $\Bc$ is a sheaf of rings on $X$ endowed with a morphism of sheaves of rings $\OO\to\Bc$, such that $\Bc$ is quasi-coherent as an $\OO$-module. Then $(X,\Bc)$ is a finite space (it has flat restrictions). Moreover, the identity on $X$ and the morphism $\OO\to \Bc$ give a morphism of ringed finite spaces $(X,\Bc)\to (X,\OO)$. A $\Bc$-module $\M$ is quasi-coherent if and only if it is quasi-coherent as an $\OO$-module. It follows easily that $(X,\Bc)$ is schematic if and only if $(X,\OO)$ is schematic. In this case, the morphism $(X,\Bc)\to (X,\OO)$ is schematic and affine.
\end{ejem}

\begin{thm}\label{Stein-fact}(Stein's factorization). Let $f\colon X\to Y$ be a schematic morphism between schematic spaces. Then $f$ factors through an schematic morphism $f'\colon X\to Y'$ such that $f'_*\OO_X=\OO_{Y'}$ and an affine morphism $Y'\to Y$ which is the identity on the topological spaces. If $f$ is affine, then $f'$ is also affine and the functors
\[ \{ \text{Quasi-coherent }\OO_X-\text{modules}\} \overset{f'_*}{\underset{{f'}^*}{\overset\longrightarrow\leftarrow}}   \{ \text{Quasi-coherent }\OO_{Y'}-\text{modules}\} \]
are mutually inverse.
\end{thm}

\begin{proof} Let $Y'$ be the finite space whose underlying topological space is $Y$ and whose sheaf of rings is $f_*\OO_X$. The morphism $f'\colon X\to Y'$ is the obvious one ($f'=f$ as continuous maps, and $\OO_{Y'}\to f_*\OO_X$ is the identity) and the affine morphism $Y'\to Y$ is that of example \ref{ejem-affine}. It is clear that $f'_*\OO_X=\OO_{Y'}$. Finally, let us see that $f'$ is schematic. Let $\Gamma'\colon X\to X\times Y'$ be the graph of $f'$ (which coincides topologically with the graph of $f$). Then, $\RR \Gamma'_*\OO$ is  quasi-coherent over $\OO_{X\times Y'}$ because it is quasi-coherent over $\OO_{X\times Y}$ ($f$ is schematic) and $\OO_{X\times Y'}$ is quasi-coherent over $\OO_{X\times Y}$. Finally, if $f$ is affine, then $f'$ is affine because $f^{-1}(U_y)=f'^{-1}(U_y)$. The last assertion of the theorem follows from  Theorem \ref{projection-formula},  (3).
\end{proof}

\subsection{Fibered products}

We have seen how the flatness condition on a finite space yields good properties for quasi-coherent modules, which fail  for arbitrary ringed finite spaces. However, an important property is lost. While the category of arbitrary ringed finite spaces has fibered products, the subcategory of finite spaces has not. However, we shall see now that the category of schematic spaces and schematic morphisms has fibered products.

\begin{thm}\label{fiberedproduct} Let $f\colon X\to S$ and $g\colon Y\to S$ be schematic morphisms between schematic finite spaces. Then
\begin{enumerate}
\item The fibered product $X\times_S Y$ is a schematic finite space and the natural morphisms $X\times_SY\to X$, $X\times_SY\to Y$ are schematic.
\item If $f$ and $g$ are affine, then $h\colon X\times_S Y\to S$ is also affine and $h_*\OO_{X\times_SY}=f_*\OO_X\otimes_{\OO_S}g_*\OO_Y$.
\end{enumerate}
\end{thm}

\begin{proof} First of all, let us see that  $X\times_SY$ is a finite space, i.e., it has flat restrictions. Let $(x,y)\leq (x',y')$ in $X\times_S Y$. Let $s,s'$ be their images in $S$. We have to prove that $\OO_x\otimes_{\OO_s}\OO_y \to \OO_{x'}\otimes_{\OO_{s'}}\OO_{y'}$ is flat. Since $\OO_s\to\OO_{s'}$ is flat, the morphism $$\OO_x\otimes_{\OO_s}\OO_y \to (\OO_{x}\otimes_{\OO_{s}}\OO_{y})\otimes_{\OO_s}\OO_{s'}=(\OO_x\otimes_{\OO_s}\OO_{s'})\otimes_{\OO_{s'}} (\OO_y\otimes_{\OO_s}\OO_{s'})$$ is flat. Now, since $f\colon U_x\to U_s$ and $g\colon U_y\to U_s$ are schematic morphisms, one has that
$ \OO_x\otimes_{\OO_s}\OO_{s'} = \OO_{xs'}$ and $  \OO_y\otimes_{\OO_s}\OO_{s'}=\OO_{ys'}$, and then we have a flat morphism
\[ \OO_x\otimes_{\OO_s}\OO_y \to \OO_{xs'}\otimes_{\OO_{s'}}\OO_{ys'}.\]
Now, $\OO_{xs'}\otimes_{\OO_x}\OO_{x'}=\OO_{x's'}=\OO_{x'}$, and $\OO_{ys'}\otimes_{\OO_y}\OO_{y'}=\OO_{y's'}=\OO_{y'}$. Since $\OO_x\to\OO_{x'}$ and $\OO_y\to\OO_{y'}$ are flat, the morphism
\[ \OO_{xs'}\otimes_{\OO_{s'}}\OO_{ys'}\to \OO_{x'}\otimes_{\OO_x}(\OO_{xs'}\otimes_{\OO_{s'}}\OO_{ys'})\otimes_{\OO_y}\OO_{y'}=\OO_{x'}\otimes_{\OO_{s'}}\OO_{y'}\]
is flat. In conclusion, $\OO_x\otimes_{\OO_s}\OO_y \to \OO_{x'}\otimes_{\OO_{s'}}\OO_{y'}$ is flat.

Let us prove the rest of the theorem by induction on $\#(X\times Y)$. If $X$ and $Y$ are punctual, it is immediate. Assume the theorem holds for $\#(X\times Y)<n$, and let us assume now that $\# (X\times Y)=n$.

Let us denote $Z=X\times_SY$, and $\pi\colon Z\to X$, $\pi'\colon Z\to Y$ the natural morphisms. Whenever we take $z_i\in Z$, we shall denote by $x_i$, $y_i$ the image of $z_i$ in $X$ and $Y$.

(1) Assume that $X=U_x$, $Y=U_y$ and $f(x)=g(y)$. Let us denote $s=f(x)=g(y)$, $z=(x,y)\in Z$. Obviously $Z=U_z$. Let $\delta\colon Z\to Z\times Z$ be the diagonal morphism. Let us see that $R^i\delta_*\OO_Z = 0$ for $i>0$. Indeed, $(R^i\delta_*\OO)_{(z,z)}=0$ because $U_z$ is acyclic. Now, if $(z_1,z_2)\in Z\times Z$ is different from $(z,z)$, then, $U_{z_1z_2}=U_{x_1x_2}\times_{U_s}U_{y_1y_2}$, and $U_{x_1x_2}\times U_{y_2y_2}$ has smaller order than $U_x\times U_y$. Moreover, $U_{x_1x_2}$ and $U_{y_1y_2}$ are affine because $U_x$ and $U_y$ are schematic. By induction, $U_{z_1z_2}$ is affine, hence $(R^i\delta_*\OO)_{(z_1,z_2)}=0$. Let us see now that $\delta_*\OO$ is quasi-coherent. We have to prove that $\OO_{z_1}\otimes_{\OO_z}\OO_{z_2}=\OO_{z_1z_2}$ for any $z_1> z< z_2$. By induction, we have that $$\OO_{z_1z_2}=\OO_{x_1x_2}\otimes_{\OO_s}\OO_{y_1y_2},\quad \OO_{z_1}=\OO_{x_1}\otimes_{\OO_s}\OO_{y_1},\quad \OO_{z_2}=\OO_{x_2}\otimes_{\OO_s}\OO_{y_2}$$
One concludes easily because $\OO_{x_1x_2}=\OO_{x_1}\otimes_{\OO_s}\OO_{x_2}$ and $\OO_{y_1y_2}=\OO_{y_1}\otimes_{\OO_s}\OO_{y_2}$.
Hence $Z$ is schematic. Let us see that $\pi\colon Z\to X$ is schematic (hence affine). By Theorem \ref{sch-preserv-qc2} it suffices to see that $\RR \pi_*\M$ is quasi-coherent for any quasi-coherent module $\M$ on $Z$. Now, $R^i\pi_*\M=0$ for $i>0$, since $(R^i\pi_*\M)_x=0$ because $Z$ is acyclic and $(R^i\pi_*\M)_{x'}=0$ for $x'>x$ because $U_{x'}\times_{U_s}U_y$ is affine by induction. Finally, $\pi_*\M$ is quasi-coherent: indeed, $(\pi_*\M)_x\otimes_{\OO_x}\OO_{x'}=(\pi_*\M)_{x'}$ by Theorem \ref{aciclico}, because $U_{x'}\times_{U_s}U_y$ is, by induction, an acyclic open subset of $U_z$. Analogously, $\pi'\colon Z\to Y$ is schematic and affine. Finally, $\OO_Z(Z)=\OO_X(X)\otimes_{\OO_S(S)}\OO_Y(Y)$ because $\OO_z=\OO_x\otimes_{\OO_s}\OO_y$.

(2) Assume that $X=U_x$ (or $Y=U_y$). Let $s=f(x)$. Then $X\times_SY=U_x\times_{U_s}g^{-1}(U_s)$. If $g^{-1}(U_s)\neq Y$ we conclude by induction. So assume $Y=g^{-1}(U_s)$. If $Y=U_y$, we conclude by (1) if $g(y)=s$ or by induction if $g(y)\neq s$. If $Y$ has not a minimum,  for any $z\in Z$, $U_z$ is schematic by induction, hence $Z$ is schematic. Moreover $\pi'\colon Z\to Y$ is schematic and affine because it is schematic and affine on any $U_y$ by induction. Let us see now that $\pi\colon Z\to U_x$ is schematic. It suffices to see that $\RR \pi_*\M$ is quasi-coherent for any quasi-coherent module $\M$ on $Z$. We have to prove that
\[ (R^i\pi_*\M)_x\otimes_{\OO_x}\OO_{x'}\to (R^i\pi_*\M)_{x'}\] is an isomorphism for any $x'\in U_x$. By induction, for any $y\in Y$ one has \[ \M(U_x\times_{U_s}U_y)\otimes_{\OO_x}\OO_{x'}= \M(U_{x'}\times_{U_s}U_y).\]
Let us denote $Z'=U_{x'}\times_{U_s}U_y$, $\pi''$ the restriction of $\pi'$ to $Z'$ and $\M'$ the restriction of $\M$ to $Z'$; then,
 $\Gamma(Y, C^r\pi'_*\M)\otimes_{\OO_x}\OO_{x'} = \Gamma(Y, C^r\pi''_*\M')$. Now, since $\pi'$ and $\pi''$ are affine,
\[ \aligned H^i(U_x\times_{U_s}Y,\M)\otimes_{\OO_x}\OO_{x'}&= H^i(Y,\pi'_*\M)\otimes_{\OO_x}\OO_{x'} = H^i\Gamma (Y,C^\punto \pi'_*\M)\otimes_{\OO_x}\OO_{x'}\\ &= H^i\Gamma (Y,C^\punto \pi''_*\M')= H^i(Y,\pi''_*\M')=  H^i(U_{x'}\times_{U_s}Y,\M).\endaligned\]
Since $H^i(U_{x'}\times_{U_s}Y,\M)=(R^i\pi_*\M)_{x'}$, we have proved that $R^i\pi_*\M$ is quasi-coherent. If $Y$ is affine, then $U_x\times_{U_s}Y$ is affine because $U_x\times_{U_s}Y\to Y$ is schematic and affine, and $\pi'_*\OO_Z= g^*f_*\OO_X$, since this equality holds taking the stalk at any $y\in Y$. Taking global sections, one obtains $\OO_Z(Z)=\OO_x\otimes_{\OO_s}\OO_Y(Y)$.

(3) For general $X$ and $Y$. For any $z=(x,y)\in Z$, $U_z$ is schematic by (2), hence $Z$ is schematic. The morphisms $Z\to X$ and $Z\to Y$ are schematic because their are so on $U_x$ and $U_y$ respectively (by (2)). If $X$ and $Y$ are affine over $S$, then $Z$ is affine over $X$ and $Y$, because it is so over $U_x$ and $U_y$ respectively. Finally, if $X$ and $Y$ are affine over $S$, then $h_*\OO_Z=f_*\OO_X\otimes_{\OO_S}g_*\OO_Y$; indeed, the question is local on $S$, so we may assume that $S=U_s$. Hence $X$ and $Y$ are affine, and then $\pi'_*\OO_Z=g^*f_*\OO_X$, since this holds taking fibre at any $y\in Y$. Taking global sections, we obtain $\OO_Z(Z)=\OO_X(X)\otimes_{\OO_s}\OO_Y(Y)$.

\end{proof}

An easy consequence of this theorem is the following:

\begin{cor} Let $f\colon X\to S$ and $g\colon S'\to S$ be schematic morphisms between schematic spaces. If $f$ is affine, then $f'\colon X\times_SS'\to S'$ is affine. If in addition $f_*\OO_X=\OO_S$, then $f'_*\OO_{X\times_SS'}=\OO_{S'}$.
\end{cor}

\medskip
\subsection{Characterization of the category of schematic spaces and schematic morphisms}
\medskip

Let $\C_{Schematic}$ be the category of schematic finite spaces and schematic morphisms and $\C_{FinSp}$ the category of finite spaces (an arbitrary morphisms of ringed spaces).

\begin{thm} \label{characterization-SchematicCategory1} Let $\C$ be a subcategory of $\C_{FinSp}$. Assume that
\begin{enumerate} \item For any morphism $f\colon X\to Y$ in $\C$, $\RR f_*$ preserves quasi-coherence.
\item If $X$ belongs to $\C$, then any open subset $U$ of $X$ belongs to $\C$ and the inclusion morphism $U\hookrightarrow X$ is a morphism in $\C$.
\end{enumerate}
Then $\C$ is a subcategory of $\C_{Schematic}$. In other words, $\C_{Schematic}$ is the biggest subcategory of $\C_{FinSp}$ satisfying {\rm (1)} and {\rm (2)}.
\end{thm}

\begin{proof} If $X$ belongs to $\C$, then for any open subset $j\colon U\hookrightarrow X$, $\RR j_*$ preserves quasi-coherence (by conditions (1) and (2)). Hence, $X$ is schematic. If $f\colon X\to Y$ is a morphism in $\C$, then $\RR f_*$ preserves quasi-coherence by (1) and $X$ is schematic. Hence $f$ is schematic by Theorem \ref{sch-preserv-qc2}.
\end{proof}

\begin{thm} \label{characterization-SchematicCategory2} Let $\C$ be a subcategory of $\C_{FinSp}$. Assume that
\begin{enumerate} \item For any morphism $f\colon X\to Y$ in $\C$, $\RR f_*$ preserves quasi-coherence.
\item $\C$ is closed under products and graphs, i.e.: if $X$ and $Y$ belong to $\C$, then $X\times Y$ belongs to $\C$, and if $f\colon X\to Y$ is a morphism in $\C$, then the graph $\Gamma\colon X\to X\times Y$ is a morphism in $\C$.
\end{enumerate}
Then $\C$ is a subcategory of $\C_{Schematic}$. In other words, $\C_{Schematic}$ is the biggest subcategory of $\C_{FinSp}$ satisfying {\rm (1)} and {\rm (2)}
\end{thm}

\begin{proof}  If $f\colon X\to Y$ is a morphism in $C$, then its graph $\Gamma\colon X\to X\times Y$ is a morphism in $\C$ (by (2)), so $\RR\Gamma_*$ preserves quasi-coherence. Thus $f$ is schematic. In particular, for any object $X$ in $\C$, the identity is schematic, i.e., $X$ is schematic.

\end{proof}

\section{From Finite Spaces to Schemes}\label{Schematic-Schemes}

We have already seen that the category of punctual ringed spaces is equivalent to the category of affine schemes. Explicitly, we have the functor
\[ \aligned \Spec\colon \{\text{Punctual ringed spaces}\} &\to \{\text{Affine schemes}\}\\ (*,A) &\mapsto \Spec A\endaligned \] whose inverse is the functor $\Spec A\mapsto \Gamma(\Spec A,\OO_{\Spec A})$. Now we see how to extend this functor from finite spaces to ringed spaces.
\medskip

\subsection{The $\Spec$ functor }\label{Spec}
\medskip

 Let $(X,\OO)$ be a ringed finite space. For each $p\in X$ let us denote $S_p$ the affine scheme $S_p=\Spec\OO_p$. For
each $p\leq q$, we have a morphism of schemes $S_q\to S_p$, induced by the ring homomorphism $\OO_p\to \OO_q$. We shall define
\[ \Spec (X):=\underset{p\in X}\ilim S_p\] where $\ilim$ is the direct limit (in the category of ringed spaces). More precisely: for each $p\leq q$,
let us denote $S_{pq}=S_q$. We have morphisms $S_{pq}\to S_q$ (the identity) and $S_{pq}\to S_p$ (taking spectra in the morphism $r_{pq}\colon \OO_p\to \OO_q$). We have then morphisms
\[ \underset{p\leq q}\coprod S_{pq}  \aligned \,_{\longrightarrow} \\ \,^{\longrightarrow} \endaligned \, \underset{p}\coprod S_p\] and we define the ringed space $\Spec (X)$ as the cokernel. That is, $\Spec (X)$ is the cokernel topological space, and $\OO_{\Spec (X)}$ is the sheaf of rings defined by: for any open subset $V$ of $\Spec (X)$, we define $\OO_{\Spec (X)}(V)$ as the kernel of
\[ \underset p\prod \OO_{S_p}(V_p)   \aligned \,_{\longrightarrow} \\ \,^{\longrightarrow} \endaligned \, \underset{p\leq q}\prod \OO_{S_{pq}}(V_{pq})\] where $V_p$ (resp. $V_{pq}$) is the preimage of $V$ under the natural map $S_p\to \Spec(X)$ (resp. $S_{pq}\to \Spec (X)$).
 In particular, $X$ and $\Spec (X)$ have the same global functions, i.e., $\OO_{\Spec (X)}(\Spec (X))=\OO_X(X)$. We say that $\Spec (X)$ is {\it the ringed space obtained by gluing the affine schemes $S_p$ along the schemes $S_{pq}$.} By definition of a cokernel (or a direct limit), for any ringed space $(T,\OO_T)$, the sequence
\[ \Hom(\Spec (X),T)\to \underset {p\in X}\prod \Hom(S_p,T)\aligned \,_{\longrightarrow} \\ \,^{\longrightarrow} \endaligned \, \underset{p\leq q}\prod \Hom (S_{pq},T)\] is exact.

This construction is functorial: if $f\colon X'\to X$ is a morphism between ringed finite spaces, it induces a morphism $\Spec(f)\colon \Spec (X')\to \Spec (X)$. In particular,  for any ringed finite space $X$, the natural morphism $X\to (*,A)$ (with $ A=\OO(X)$) induces a morphism
\[ \Spec (X)\to \Spec A.\]

From now on  all finite spaces are assumed to be schematic and all morphisms are assumed to be schematic.

\begin{defn} We say that a schematic space $X$ has {\it open restrictions} if for any $p\leq q$ the morphism $\Spec\OO_q\to\Spec\OO_p$ is an open immersion.
\end{defn}

If $X$ has open restrictions, then the well known gluing technique of schemes tells us that $\Spec (X)$ is a (quasi-compact and quasi-separated) scheme. If $f\colon X\to Y$ is a schematic morphism between schematic spaces with open restrictions, then $\Spec(f)\colon \Spec X\to\Spec Y$ is a morphism of schemes. If $X\to S$, $Y\to S$ are schematic morphisms between schematic spaces with open restrictions, then $X\times_SY$ is a schematic space with open restrictions too.

If $X$ is the ringed finite space associated to a (locally affine) finite covering $\U$ of a quasi-compact and quasi-separated scheme $S$, then $X$ has open restrictions and there is a natural isomorphism $\Spec (X)\overset\sim\to S$.

Let us denote by $\C_{qcqs-Schemes}$ the category of quasi-compact and quasi-separated schemes and by  $C_{Schematic}$ the category of schematic finite spaces and schematic morphisms. Let  $\C_{Schematic}^{open}$ be the full subcategory of $\C_{Schematic}$ whose objects are those schematic spaces that have open restrictions.

\begin{prop}\label{fullessentiallysurjective} The functor
\[\Spec\colon \C_{Schematic}^{open}\to \C_{qcqs-Schemes}\] is  essentially surjective. Moreover, for any morphism of schemes $f\colon \overline S\to S$, there exists a morphism $h\colon\overline X\to X$ in $\C_{Schematic}^{open}$ such that $S\simeq\Spec(X)$, $\overline S\simeq \Spec(\overline X)$ and, via these isomorphisms, $\Spec(h)=f$.
\end{prop}

\begin{proof}  Given a quasi-compact and quasi-separated scheme $S$, choose a (locally affine) finite covering $\U$ of $S$ and let $X$ be the associated finite space. Then $X$ has open restrictions and $\Spec(X)$ is canonically isomorphic to $S$.

Let $f\colon S\to \overline S$ be a morphism of schemes, and let $\U$, $\overline \U$ be (locally affine) finite coverings of $S$ and $\overline S$, such that $\U$ is thinner than $f^{-1}(\overline\U)$. If $X$ and $\overline X$ are the associated finite spaces, one has a morphism $h\colon X\to \overline X$ such that $\Spec(h)=f$.
\end{proof}

\subsection{Localization by weak equivalences}$\,$
\medskip

The functor $\Spec\colon \C_{Schematic}^{open}\to \C_{qcqs-Schemes}$ is not faithful, because the finite spaces associated to different coverings of a scheme $S$ are not isomorphic (not even homotopic). In order to avoid this problem, i.e., in order to identify two finite spaces constructed from different coverings of the same scheme, we introduce the notion of a weak equivalence.

\begin{defn} An affine schematic morphism $f\colon X\to Y$ such that $f_*\OO_X=\OO_Y$ is called a {\it weal equivalence}.
\end{defn}

A weak equivalence $f\colon X\to Y$ yields an equivalence between the categories of quasi-coherent modules on $X$ and $Y$, by Theorem \ref{Stein-fact}. A schematic morphism $f\colon X\to Y$ is a weak equivalence if and only if for any open subset $V$ of $Y$, $f\colon f^{-1}(V)\to V$ is a weak equivalence. If $f\colon X\to Y$ is a weak equivalence, then, for any schematic morphism $Y'\to Y$, the induced morphism $X\times_YY'\to Y'$ is a weak equivalence. Composition of weak equivalences is a weak equivalence. If $X$ is affine and $A=\OO(X)$, the natural morphism $X\to \{ *,A\}$ is a weak equivalence.

\begin{ejem}\label{ejem-qc-iso} Let $S$ be a quasi-compact and quasi-separated scheme, $\U=\{ U_i\}$ and $\U'=\{ U'_j\}$ two (locally affine) finite coverings of $S$. Assume that  $\U'$ is thinner than $S$ (this means that, for any $s\in S$, $U'_s\subseteq U_s$). Let $\pi'\colon S\to X'$ and $\pi\colon S\to X$ be the finite spaces associated to $\U'$ and $\U$. Since $\U'$ is thinner that $\U$, one has a morphism $f\colon X'\to X$ such that $f\circ \pi'=\pi$. Then $f$ is a weak equivalence.
\end{ejem}

Let us denote by $\C_{Schematic}^{open}[\W^{-1}]$ the localization of the category $\C_{Schematic}^{open}$ by weak equivalences.
A morphism $X\to Y$ in $\C_{Schematic}^{open}[\W^{-1}]$ is represented by a diagram $$\xymatrix{ & T\ar[dl]_\phi \ar[rd]^f & \\  X &  & Y  }$$
where $\phi$ is a weak equivalence and $f$ a schematic morphism. Two diagrams
$$\xymatrix{ & T\ar[dl]_\phi \ar[rd]^f & \\  X &  & Y  } \qquad \xymatrix{ & T'\ar[dl]_{\phi'} \ar[rd]^{f'} & \\  X &  & Y  }$$ are equivalent (i.e. they represent the same morphism in $\C_{Schematic}^{open}[\W^{-1}]$) if there exist a schematic space $T''$ with open restrictions and weak equivalences $\xi\colon T''\to T$, $\xi'\colon T''\to T'$ such that the diagram
\[ \xymatrix{ & & T''\ar[dl]_\xi \ar[dr]^{\xi'}\\ & T \ar[dl]_\phi \ar[drrr]^f  & & T' \ar[dlll]_{\phi'} \ar[dr]^{f'} &  \\ X & & & & Y
}\] is commutative. We denote by $f/\phi\colon X\to Y$ the morphism in $\C_{Schematic}^{open}[\W^{-1}]$ represented by $f$ and $\phi$. The composition of two morphisms
\[\xymatrix{ & T \ar[dl]_\phi \ar[rd]^f & & T'\ar[dl]_\psi \ar[rd]^g & \\ X & & Y & & Z
}\] is given by $$\xymatrix{ & T\times_YT'\ar[dl]_\xi \ar[rd]^h & \\  X &  & Y  }$$ where $\xi$ (resp. $h$) is the composition of $\phi$ (resp. $g$) with the natural morphism $T\times_YT'\to T$ (resp.  the natural morphism $T\times_YT'\to T$). Notice that $T\times_YT'\to T$ is a weak equivalence because $\psi$ is a weak equivalence; hence $\xi$ is a weak equivalence.

Our aim  now is to show that the functor $\Spec\colon \C_{Schematic}^{open}\to  \C_{qcqs-Schemes}$ factors through the localization  $\C_{Schematic}^{open}[\W^{-1}]$.

\begin{prop}\label{S(affine)} Let $X$ be an affine and schematic finite space, $A=\OO(X)$. Then, the natural morphism of ringed spaces
\[ \Spec (X)\to \Spec A\] is an isomorphism.
\end{prop}

\begin{proof} Let us denote $B=\underset{p\in X}\prod \OO_p$ and $A\to B$ the natural (injective) morphism. We know that  this morphism is faithfully flat (Proposition \ref{derivedcat-affine}). By faithfully flat descent,  we have an exact sequence (in the category of schemes)
\[ \Spec(B\otimes_AB) \aligned \,_{\longrightarrow} \\ \,^{\longrightarrow} \endaligned \,\Spec B\to \Spec A \] which is also an exact sequence in the category or ringed spaces (it is an easy exercise). Now, following the notations of subsection \ref{Spec}, one has $\Spec (B)=\underset{p\in X}\coprod \Spec\OO_p= \underset{p\in X}\coprod  S_p$, and $\Spec (B\otimes_AB)=\underset{p,q\in X}\coprod \Spec (\OO_p\otimes_A\OO_q) = \underset{p,q\in X}\coprod \Spec \OO_{pq}$, where the last equality is due to Proposition \ref{local-structure-affine}. That is, $\Spec A$ is obtained by gluing the schemes $S_p$ along the schemes $\Spec\OO_{pq}$.

Now, $U_{pq}$ is affine, because  $X$ is affine and schematic, hence semi-separated. Then, the morphism $\OO_{pq}\to\underset{t\in U_{pq}}\prod \OO_t$ is faithfully flat, so $ \underset{t\in U_{pq}}\coprod S_t \to \Spec\OO_{pq}$ is an epimorphism. We have then an exact sequence
\[ \underset{\underset{t\in U_{pq}} { \,_{p,q\in X}}}\coprod S_t  \aligned \,_{\longrightarrow} \\ \,^{\longrightarrow} \endaligned \, \underset{p\in X}\coprod S_p\to \Spec A\]

Notice that $S_t=S_{pt}=S_{qt}$ for any $t\in U_{pq}$. It is now clear that gluing the schemes $S_p$ along the schemes $S_{pq}=\Spec (\OO_{pq})$ (with arbitrary $p,q$) is the same as gluing the schemes $S_p$ along the schemes $S_{pq}$ (with $p\leq q$). This says that $\Spec A=\Spec (X)$.
\end{proof}

\begin{prop}\label{S(qc)} If $f\colon X\to Y$ is a weak equivalence, then $\Spec(f)\colon \Spec (X)\to \Spec (Y)$ is an isomorphism.
\end{prop}

\begin{proof} We have to prove that $\Spec$ transforms weak equivalences into isomorphisms. Let $f\colon X\to Y$ be a weak equivalence, and let us see that $\Spec (f)\colon \Spec (X)\to \Spec (Y)$ has an inverse $h\colon \Spec (Y)\to \Spec (X)$. For each $y\in Y$, let us denote $X_y=f^{-1}(U_y)$ and $f_y\colon X_y\to U_y$. We have that  $X_y$ is affine (because $f$ is affine)  and $\OO_{X_y}(X_y)=\OO_y$, because $f_*\OO_X=\OO_Y$. By Proposition \ref{S(affine)}, $\Spec (f_y)\colon \Spec (X_y)\to \Spec (U_y)$ is an isomorphism. Hence we have a morphism
\[ h_y\colon S_y=\Spec (U_y)\to \Spec (X)\] defined as the composition of
$\Spec (f_y)^{-1}\colon \Spec (U_y)\to \Spec (X_y)$ with the natural morphism $\Spec (X_y)\to \Spec(X)$. For any $y\leq y'$, let $r_{yy'}^*\colon S_{y'}\to S_y$ be the morphism induced by $r_{yy'}\colon \OO_y\to\OO_{y'}$. From the commutativity of the diagram

$$\xymatrix{ \Spec (X_{y'})\ar[r]\ar[d]  & \Spec (U_{y'})  \ar[d] \\ \Spec (X_{y})\ar[r]\ar[d]  & \Spec (U_y)  \\ \Spec (X) & }$$ it follows that  $h_y\circ r_{yy'}^*=h_{y'}$, hence one has a morphism $h\colon \Spec (Y)\to \Spec (X)$. It is clear that $h$ is an inverse of $\Spec (f)$.
\end{proof}

This Proposition tells us that the functor $\Spec\colon \C_{Schematic}^{open}\to   \C_{qcqs-Schemes} $ factors through $\C_{Schematic}^{open}[\W^{-1}]$. Hence we obtain a functor \[ \Spec\colon \C_{Schematic}^{open}[\W^{-1}]\to \C_{qcqs-Schemes} \]
and we can state the main result of this section:

\begin{thm}\label{embedding-schemes} The functor  $$\Spec\colon \C_{Schematic}^{open}[\W^{-1}]\to \C_{qcqs-Schemes}$$  is fully faithful and essentially surjective.
\end{thm}

\begin{proof}  By Proposition \ref{fullessentiallysurjective}, it remains to prove that it is faithful. It is enough to prove that if $f_1,f_2\colon X\to Y$ are two schematic morphisms between schematic spaces with open restrictions such that $\Spec(f_1)=\Spec(f_2)$, then $f_1$ and $f_2$ are equivalent in the localization $\C_{Schematic}^{open}[\W^{-1}]$. We may assume that $X$ and $Y$ are $T_0$.

Let $X$ be a $T_0$-schematic space with open restrictions. For each $x\in X$, $S_x=\Spec\OO_x$ is an open subscheme of $\Spec (X)$.
Thus,  $\U=\{ S_x\}_{x\in X}$ is an open covering of $\Spec (X)$. Let $X'$ the finite space associated to $\U$, and $\pi\colon\Spec(X)\to X'$ the natural morphism. Let us see that there is a natural weak equivalence $f_X\colon X'\to X$. First notice that, for any $p,q\in X$, $S_p\cap S_q=\underset{t\in U_{pq}}\cup S_t$. Let $\T_\U$  be the (finite) topology of $\Spec (X)$ generated by $\U$ and $\T_X$ the topology of $X$. Let us consider the map
\[  \begin{aligned}\psi\colon \T_X &\to \T_\U\\ U&\mapsto \psi(U)=\underset{t\in U}\cup S_t\end{aligned}   \]
It is clear that $\psi(U\cup V)=\psi(U)\cup \psi(V)$  and $\psi(U_p)=S_p$. Moreover, $\psi(U\cap V)= \psi(U)\cap \psi( V)$; indeed, since any $U$ is a union of $U_p's$ and $\psi$ preserves unions, we may assume that $U=U_p$ and $V=U_q$, and then the result follows from the equality  $S_p\cap S_q=\underset{t\in U_{pq}}\cup S_t$. In conclusion, $\psi$ is a  morphism of distributive lattices. Hence, it induces a continuous map  $f_X\colon X'\to X$, whose composition with $\pi$ is a continous map $\pi_X\colon\Spec (X)\to X$  such that $\pi_X^{-1}(U_p)=\psi(U_p)=S_p$. Since $\OO_p=\Gamma (S_p,\OO_{S_p})$,  one has a natural isomorphism $\OO_X\to {f_X}_*\OO_{X'}={\pi_X}_*\OO_{\Spec (X)}$, such that $f_X$ is a morphism of ringed spaces. Let us see that $f_X$ is a weak equivalence.  Let us see that $f_X$ is schematic: by Theorem \ref{sch-preserv-qc2}, it suffices to see that $\RR {f_X}_*\Nc$ is quasi-coherent for any $\Nc$ quasi-coherent module on $X'$. Now, $\Nc=\pi_*\M$ for some quasi-coherent module $\M$ on $\Spec (X)$, so  $\RR {f_X}_*\Nc=\RR {\pi_X}_*\M$ (because $R^i\pi_*\M=0$ for $i>0$) and $\RR {\pi_X}_*\M={\pi_X}_*\M$ because $\pi_X^{-1}(U_p)=S_p$ is an affine scheme. Finally, the quasi-coherence of ${\pi_X}_*\M$ is also a consequence of the equality $\pi_X^{-1}(U_p)=S_p$ and the hypothesis that $S_q\to S_p$ is an open immersion. Let us conclude that $f_X$ is a weak equivalence; $f_X$ is affine: $f_X^{-1}(U_p)$ is affine because $\pi^{-1}(f_X^{-1}(U_p))=S_p$ is affine. Since  ${f_X}_*\OO_{X'}=  \OO_X$, we are done.

Now, let $f_1,f_2\colon X\to Y$ be two schematic morphisms between $T_0$-schematic spaces with open restrictions, such that $\Spec(f_1)=\Spec (f_2)$. Let us denote $h=\Spec(f_i)$. Since $\{ S_x\}$ is thinner that $\{ h^{-1}(S_y)\}$, it induces a morphism $h'\colon X'\to Y'$ and one has a commutative  diagram
$$\xymatrix{
%\Spec (X)\ar[r]^{h} \ar[d]_{\pi} & \Spec (Y)\ar[d]^{\pi} \\
X'\ar[r]^{h'} \ar[d]_{f_X}  & Y'\ar[d]^{f_Y} \\ X \ar[r]^{f_i} & Y}$$
so $f_1$ and $f_2$ are equivalent (since they both are  equivalent to $h'$).
\end{proof}

Let us denote by $\C_{AffSchSp}$ the localization of the category of affine schematic spaces (and schematic morphisms) by weak equivalences. Let us see that $\C_{AffSchSp}$ is equivalent to the category of affine schemes.

\begin{thm}\label{affinescheme-affinespace} The functor $$\aligned \Phi\colon \C_{AffineSchemes}&\to \C_{AffSchSp} \\ \Spec A&\mapsto (*,A)\endaligned$$ is an equivalence.
\end{thm}
\begin{proof} If $X$ is an affine schematic space and $A=\OO(X)$, then $X\to (*,A)$ is a weak equivalence and $\Spec(X)=\Spec A$ (Proposition \ref{S(affine)}). Now, for any affine schematic spaces $X$ and $Y$, with global functions $A$ and $B$, one has:
\[ \Hom_{\C_{AffSchSp} }(X,Y)=\Hom_{\C_{AffSchSp} }((*,A),(*,B)) \]
and it is clear that $\Hom_{\C_{AffSchSp} }((*,A),(*,B))=\Hom_{\text{rings}}(B,A)=\Hom_{\text{schemes}}(\Spec A,\Spec B)$.
One concludes that the functor $\C_{AffSchSp} \to \C_{AffineSchemes}$, $X\mapsto \Spec \OO(X)$, is an inverse of $\Phi$.
\end{proof}

\begin{rem} (A schematic space which is not a scheme). Let $(X,\OO)$ be the following finite space: it has two closed points $p,q$ and one generic point $g$; the sheaf$\OO$ is given by  $\OO_p=\OO_q=k[x]$, $\OO_g=k(x)$, and the morphisms $r_{pg}$, $r_{qg}$ are the natural inclusions. Then $X$ is an schematic space of dimension 1. One can easily calculate its cohomology:
\[ H^0(X,\OO)=k[x],\qquad H^1(X,\OO)=k(x)/k[x]\]
Now, $\Spec (X)$ is the gluing of two affine lines (i.e., $\Spec k[x]$) along their generic point. This is not a scheme (though it is a locally ringed space).

\end{rem}

\end{document}